\documentclass[12pt,reqno]{amsart}
\usepackage{amsmath,amssymb,amsthm,mathrsfs}
\allowdisplaybreaks
\usepackage{graphicx,cite}
\usepackage{newtxtext,newtxmath}
\usepackage{cases}
\usepackage{color}
\usepackage{appendix}
\usepackage{enumerate}
\usepackage{bm}
\usepackage{mathtools}
\usepackage[ruled]{algorithm2e}
\usepackage{verbatim}
\usepackage{subfigure} 
\usepackage{soul, xcolor}
\usepackage{float}
\usepackage[colorlinks, linkcolor=blue, citecolor=blue]{hyperref}
\graphicspath{{figures/}}
\usepackage{epstopdf}
\usepackage{rotating}
\usepackage{textcomp}
\usepackage{cleveref}
\crefrangelabelformat{equation}{(#3#1#4-#5#2#6)}
\allowdisplaybreaks[4] 
\newtheorem{theorem}{Theorem}[section]

\newtheorem{lemma}[theorem]{Lemma}
\newtheorem{proposition}[theorem]{Proposition}
\newtheorem{remark}[theorem]{Remark}

\renewcommand{\Re}{\operatorname{Re}}

\numberwithin{equation}{section}

\begin{document}

\title[Limiting Absorption Principle in Multilayer Spheres ] {Limiting Absorption Principle for the Helmholtz Equation with Sign-Changing Coefficients in Multilayer Spheres}

\author{Wenjing Zhang}
\address{School of
Mathematics and Statistics,  Northeast Normal University, Changchun, Jilin 130024, China }
\email{zhangwj034@nenu.edu.cn}
	
\author{Yixian Gao}
\address{School of
Mathematics and Statistics, Center for Mathematics and
Interdisciplinary Sciences, Northeast Normal University, Changchun, Jilin 130024, China. }
\email{gaoyx643@nenu.edu.cn}

\thanks{  The research of Y. Gao was supported by the NSFC
(project number, 12371187) and Science and Technology Development Plan Project of Jilin Province (project number, 20240101006JJ)}
	
\subjclass[2020]{35J05, 35P25, 35R05}

\keywords{Helmholtz equation, limiting absorption principle, sign-changing coefficients, $\mathbb T$-coercivity, multilayer metamaterials, Dirichlet-to-Neumann operator
}

\begin{abstract}
This paper investigates a multilayered Helmholtz model in $\mathbb{R}^d$ ($d \ge 2$) characterized by concentric layers of materials with alternating positive and negative refractive indices. To overcome the loss of coercivity induced by the sign-changing material parameters, we construct a bespoke $\mathbb T$-coercivity operator to restore the coercive structure of the problem. Furthermore, to address the inherent lack of compactness on unbounded domains, we integrate a complex-wavenumber Dirichlet-to-Neumann (DtN) operator into this framework. By combining this variational synthesis with sharp \textit{a priori} estimates, we rigorously establish the limiting absorption principle and prove the well-posedness of the corresponding transmission problem in appropriate function spaces. Crucially, we quantify the dependence of uniqueness on the domain geometry by explicitly analyzing the optimal trace constant, thereby providing a rigorous mathematical criterion for the design of multi-layer metamaterials.
\end{abstract}

\maketitle
	
\section{introduction}
In the context of time-harmonic scattering on unbounded domains, the operator $H - \lambda I$ inherently fails to admit a bounded inverse on the standard $L^2$ space whenever the spectral parameter $\lambda$ resides within its continuous spectrum. Consequently, the physically relevant outgoing solution cannot be uniquely characterized within the standard $L^2$ framework. To isolate the steady-state solution satisfying the Sommerfeld radiation condition, a classical and rigorous approach involves regularizing the spectral parameter via a small absorption term, replacing $\lambda$ with $\lambda + i\varepsilon$. By analyzing the corresponding operator and passing to the limit as $\varepsilon \to 0^+$, one selects the physical radiating solution---a procedure  known as the limiting absorption principle (LAP).

The genesis of selecting physical radiating solutions via infinitesimal absorption traces back at least to Ignatowsky \cite{v1905reflexion}. Sve\v{s}nikov \cite{MR37195} formally articulated this concept as the radiation principle, applying it to waveguide problems.
The mathematical foundation was substantially advanced by \`E\u{\i}dus  \cite{MR145187} in the early 1960s in connection with wave propagation  problems. In the 1970s, Agmon's influential work \cite{MR397194}  formulated the limiting absorption principle as a powerful resolvent  framework for Schr\"odinger operators, firmly establishing its role as a  fundamental tool in modern scattering theory and spectral analysis on  unbounded domains.

In recent years, the LAP has been successfully extended to accommodate increasingly complex heterogeneous media. Notable developments include the work of Kirsch and Lechleiter \cite{kirsch2018}, who employed the Floquet--Bloch transform to establish the LAP for two-dimensional periodic media, and Hu and Kirsch \cite{hu2026}, who addressed scattering by periodic curves. Furthermore, Cacciafesta et al. \cite{MR3870070} derived the principle for variable-coefficient Helmholtz equations in exterior domains via uniform \textit{a priori} estimates.

However, a distinct and mathematically substantial challenge arises when the coefficients change sign across material interfaces, a setting fundamental to the modeling of negative-index materials. In this indefinite regime, the associated sesquilinear form loses its standard coercivity, precluding the direct application of classical variational techniques. The functional analytic treatment of such problems has been profoundly shaped by the $\mathbb{T}$-coercivity framework initiated by Bonnet-Ben Dhia et.al. \cite{MR2644187, Bonnet2012, MR3200087}. Its main idea is to construct a problem-adapted bounded isomorphism \(\mathbb{T}\) so that the transformed form recovers a coercive-plus-compact structure. In the scalar setting, a key feature of this theory is the appearance of ``forbidden'' contrast ratios: when the ratio of material parameters across an interface belongs to a critical interval, depending in general on the interface geometry, the corresponding operator may fail to be Fredholm in the natural variational setting \cite{Bonnet2012}. Parallel to these developments, Nguyen \cite{MR3515306} established the LAP and well-posedness for sign-changing transmission problems through a delicate analysis of \textit{a priori} estimates for the Cauchy problem. Building upon these foundational frameworks, we previously established the LAP for Helmholtz systems with sign-changing coefficients in two-dimensional periodic structures \cite{MR5002626}, utilizing $\mathbb{T}$-coercivity to recover an appropriate coercive structure.

Beyond periodic structures, concentric multilayered spherical geometries have served as important model geometries, owing to their symmetry and analytical tractability in scattering problems \cite{CHEN20155, ZHU2020, Johnson96}. These configurations have gained renewed significance in the modern theory of metamaterials following the experimental observation of negative refraction in artificially structured media \cite{shelby2001}, which provided important experimental support for Veselago's theoretical prediction of media with simultaneously negative permittivity and permeability \cite{Veselago1968}. By appropriately tuning the material parameters and geometric scales of the constituent layers, wave propagation can be controlled at subwavelength scales through resonant or scattering-cancellation mechanisms \cite{DENG2026, PhysRevLett.100.113901}. More broadly, multilayered architectures provide versatile physical prototypes for controlling wave phenomena across electromagnetic, acoustic, and elastic regimes \cite{DENG2026, PhysRevLett.301, PhysRevLett.100.113901, PhysRevLett.108} and serve as platforms for realizing metamaterial elements with anomalous functionalities.

Motivated by these advancements, the present paper undertakes a rigorous mathematical investigation of scattering by a multilayered concentric structure comprising alternating positive- and negative-refractive-index layers. From an analytical standpoint, this leads to an unbounded transmission problem for a Helmholtz-type equation in $\mathbb{R}^d$ $(d \ge 2)$, in which the permittivity and permeability exhibit sign-changing jump discontinuities across adjacent interfaces. While the LAP has been extensively studied for wave propagation in unbounded domains, its application to concentric multilayered structures with sign-changing parameters remains largely unexplored. In this work, we establish the LAP for such sign-changing concentric models in $\mathbb{R}^d$ $(d \ge 2)$.

In this configuration, the adjacency of layers with opposite signs of material properties yields a governing operator that suffers from a severe loss of global ellipticity. To overcome this, we construct a bespoke $\mathbb{T}$-coercivity operator tailored to the concentric multi-layer geometry. The transition from periodic structures to radial stratification represents a significant mathematical leap; here, the $\mathbb{T}$ operator must be meticulously designed to restore the coercive structure while remaining compatible with the radial transmission conditions. Furthermore, to handle the lack of compactness caused by the unbounded exterior domain, we incorporate into the $\mathbb{T}$-coercivity framework the spherical DtN operator for complex wavenumbers \cite{gräßle2025}. This allows us to combine the interior $\mathbb{T}$-coercivity argument with the exact radiation condition encoded by the DtN map, and to reformulate the unbounded dissipative scattering problem as an equivalent boundary value problem on a truncated domain.

A central achievement of our mathematical analysis is the explicit quantification of how the well-posedness of this singular problem hinges upon the contrast ratio of the permittivities. We rigorously derive a sufficient condition for uniqueness that is intrinsically linked to the optimal constant of the  trace theorem. By implementing a targeted change of variables and analyzing the associated eigenvalue problems via tailored functionals, we completely characterize the dependence of these constants on the radial geometry (specifically, the shell aspect ratios). Consequently, our results provide a rigorous theoretical criterion for the design of multi-layer metamaterials: for a prescribed negative-index material contrast, the uniqueness of radiating solutions can be guaranteed via a strategic, mathematically informed selection of the layer dimensions.

The remainder of this paper is organized as follows. Section \ref{sec2} defines the problem setting and establishes the requisite notation. Section \ref{sec3} provides several auxiliary lemmas essential for the subsequent analysis. Our main results and their proofs are presented in Section \ref{sec4}. Finally, Section \ref{sec5} serves as an appendix containing supplementary material to support the analysis.

\section{Problem Formulation}\label{sec2}

We consider the Helmholtz equation in $\mathbb{R}^d$ ($d \ge 2$) within a radially stratified medium consisting of $n \ge 3$ concentric layers, where $n$ is an odd integer. The geometry is defined by a sequence of interface radii $0 < R_1 < R_2 < \dots < R_{n-1} < \infty$, which partition $\mathbb{R}^d$ into the following open domains:
\begin{itemize}
    \item the core region $D_1 := \{ x \in \mathbb{R}^d : |x| < R_1 \}$;
    \item the annular layers $D_i := \{ x \in \mathbb{R}^d : R_{i-1} < |x| < R_i \}$ for $i = 2, \dots, n-1$;
    \item the exterior domain $D_n := \{ x \in \mathbb{R}^d : |x| > R_{n-1} \}$.
\end{itemize}
The interface between $D_i$ and $D_{i+1}$ is denoted by $\Gamma_i := \partial D_i \cap \partial D_{i+1}$ for $1 \le i \le n-1$.

\begin{figure}[h]
\centering
\includegraphics[width=0.4\textwidth]{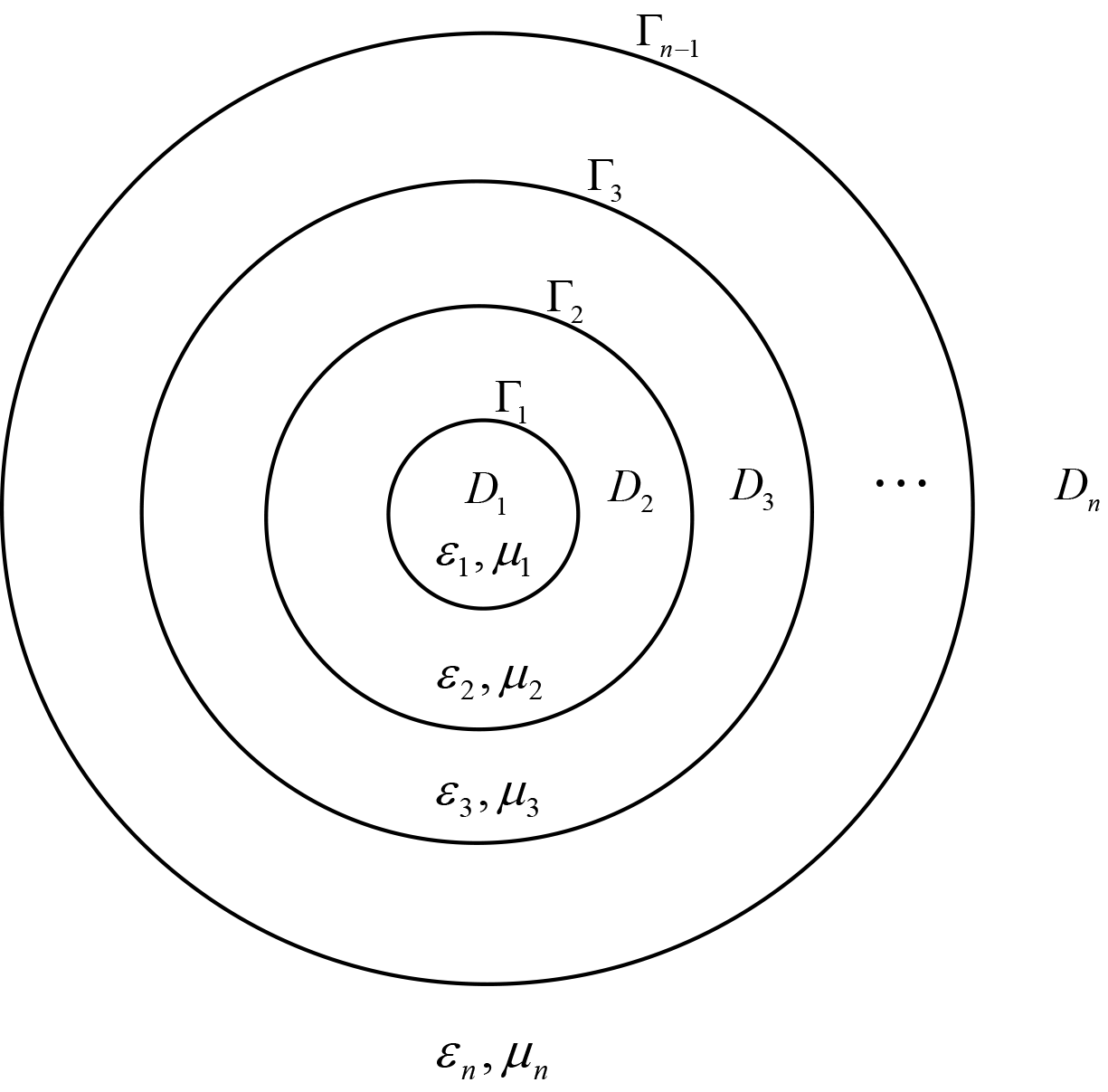}
\caption{Schematic of the problem geometry. While the two-dimensional configuration ($d=2$) is illustrated, the analysis pertains to the radially stratified structure in any dimension $d \ge 2$.}
\label{figure1}
\end{figure}

The medium consists of layers with alternating material properties. We use \(k\) to denote odd layer indices and \(l\) to denote even layer indices. More precisely, we define
\[
\mathcal{I}_{\mathrm{odd}}
:=
\{\, k\in\{1,\ldots,n\} : k \text{ is odd} \,\},
\qquad
\mathcal{I}_{\mathrm{even}}
:=
\{\, l\in\{1,\ldots,n\} : l \text{ is even} \,\}.
\]
Thus, \(k\in \mathcal{I}_{\mathrm{odd}}\) labels an odd-indexed layer, whereas \(l\in \mathcal{I}_{\mathrm{even}}\) labels an even-indexed layer. The permittivity $\varepsilon(x)$ and permeability $\mu(x)$ are piecewise constant functions defined by
\begin{equation*}
    \varepsilon(x) = \varepsilon_i, \quad \mu(x) = \mu_i, \quad x \in D_i,
\end{equation*}
where the constants satisfy $\varepsilon_i, \mu_i > 0$ for $i \in \mathcal{I}_{\mathrm{odd}}$ and $\varepsilon_i, \mu_i < 0$ for $i \in \mathcal{I}_{\mathrm{even}}$. Consequently, the material parameters exhibit jump discontinuities in sign across each interface $\Gamma_i$.

It is well known that in the case where the contrast ratio at an interface equals $-1$, i.e., $\varepsilon_i / \varepsilon_{i+1} = -1$, the system becomes ill-posed \cite{MR3200087, MR2263683}. To ensure the validity of our functional-analytic framework, we assume throughout this work that
\begin{equation*}
    \frac{\varepsilon_i}{\varepsilon_{i+1}} \neq -1 \quad \text{for all } i \in \{1, \dots, n-1\}.
\end{equation*}

We investigate the well-posedness of the following system:
\begin{equation}\label{eq11}
\begin{cases}
\nabla \cdot \big(\varepsilon^{-1} \nabla u\big) + \omega^2 \mu u = f & \text{in } \mathbb{R}^d,\\[1mm]
u \text{ satisfies the Sommerfeld radiation condition as } |x| \to \infty.
\end{cases}
\end{equation}
The source term \(f\in L^2(\mathbb{R}^d)\) is assumed to have compact support. More precisely, there exists a fixed radius \(R_0>0\) such that
\[
\operatorname{supp} f \Subset B_{R_0},
\]
where \(B_{R_0}\) denotes the open ball centered at the origin. The parameter \(\omega>0\) denotes the frequency of the time-harmonic wave.

Owing to the indefinite sign of the material parameters across the interfaces, the governing operator lacks ellipticity, which precludes the direct application of standard elliptic estimates. To overcome this, we introduce a dissipation term and consider the corresponding dissipative problem to study the limiting absorption principle. We first establish existence and uniqueness for the regularized (dissipative) problem for a fixed $\sigma > 0$ and subsequently investigate the asymptotic behavior of its solutions as $\sigma \to 0^+$. Specifically, we consider the following dissipative system:
\begin{equation}\label{eq10}
\begin{cases}
\nabla \cdot (\varepsilon^{-1}_\sigma \nabla u_\sigma) + \omega^2 \mu u_\sigma + \mathrm{i}\sigma u_\sigma = f & \text{in } \mathbb{R}^d,\\[3pt]
u_\sigma \text{ satisfies the outgoing condition as } |x| \to \infty,
\end{cases}
\end{equation}
where the regularized permittivity $\varepsilon_\sigma$ is the piecewise constant function defined by
\[
\varepsilon_\sigma(x) =
\begin{cases}
\varepsilon_k, & x \in D_k,\\
\varepsilon_l + \dfrac{\mathrm{i}\sigma}{\omega}, & x \in D_l,
\end{cases}
\]
with $\sigma > 0$ denoting a small regularization parameter.

Across each interface $\Gamma_i$, the solution $u_\sigma$ is subject to transmission conditions ensuring the continuity of the trace and the conormal derivative. These matching conditions are given by
\begin{equation}\label{eq:trans_cond}
\begin{cases}
u_\sigma|_{D_k} = u_\sigma|_{D_{k+1}}
& \text{on } \Gamma_k,\quad k = 1,3,\ldots,n-2, \\
\varepsilon_k^{-1} \nabla u_\sigma|_{D_k} \cdot\vec{n}
= \left( \varepsilon_{k+1} + \dfrac{\mathrm{i}\sigma}{\omega} \right)^{-1}
\nabla u_\sigma|_{D_{k+1}} \cdot\vec{n}
& \text{on } \Gamma_k,\quad k = 1,3,\ldots,n-2, \\
u_\sigma|_{D_l} = u_\sigma|_{D_{l+1}}
& \text{on } \Gamma_l,\quad l = 2,4,\ldots,n-1, \\
\left( \varepsilon_l + \dfrac{\mathrm{i}\sigma}{\omega} \right)^{-1}
\nabla u_\sigma|_{D_l} \cdot\vec{n}
= \varepsilon_{l+1}^{-1}
\nabla u_\sigma|_{D_{l+1}} \cdot \vec{n}
& \text{on } \Gamma_l,\quad l = 2,4,\ldots,n-1,
\end{cases}
\end{equation}
where $\vec{n}$ denotes the unit normal vector to $\Gamma_i$ oriented from $D_i$ toward $D_{i+1}$. Consistent with our previous definition of $\varepsilon_\sigma$, the flux conditions explicitly account for the dissipative regularization in the even-indexed layers.

\section{Preliminaries}\label{sec3}
Due to  the sign-changing nature of the material parameters across interfaces, we perform a localized analysis on each subdomain $D_i$ to derive uniform a priori estimates.

\begin{lemma}\label{lemm1}
For each $i \in \{1, \dots, n-1\}$, let $u_\sigma \in H^1(D_i)$ be a solution to the local problem
\begin{equation*}
\begin{cases}
\nabla \cdot (\varepsilon^{-1}_\sigma \nabla u_\sigma) + \omega^2 \mu u_\sigma + \mathrm{i}\sigma u_\sigma = f & \text{in } D_i, \\
u_\sigma \in H^{1/2}(\partial D_i), \quad \varepsilon^{-1}_\sigma \nabla u_\sigma \cdot \vec{n} \in H^{-1/2}(\partial D_i) & \text{on } \partial D_i.
\end{cases}
\end{equation*}
Then there exists a constant $C_i > 0$, independent of $\sigma$ and $f$, such that
\begin{equation*}
\|u_\sigma\|_{H^1(D_i)} \le C_i \left( \|\nabla u_\sigma \cdot \vec{n}\|_{H^{-1/2}(\partial D_i)} + \|f\|_{L^2(D_i)} + \|u_\sigma\|_{L^2(D_i)} \right).
\end{equation*}
\end{lemma}

\begin{proof}
Let $\varepsilon_{\sigma,i} := \varepsilon_\sigma|_{D_i}$ denote the local restriction of the regularized permittivity to the subdomain $D_i$. Testing the governing equation against $\overline{u_\sigma} \in H^1(D_i)$ and invoking Green’s first identity, we obtain the variational identity:
\begin{equation*}
\varepsilon_{\sigma,i}^{-1} \int_{D_i} |\nabla u_\sigma|^2 \mathrm{d}x - \omega^2 \mu_i \int_{D_i} |u_\sigma|^2 \mathrm{d}x - \mathrm{i}\sigma \int_{D_i} |u_\sigma|^2 \mathrm{d}x = \int_{\partial D_i} \varepsilon_{\sigma,i}^{-1} \nabla u_\sigma \cdot \vec{n} \, \overline{u_\sigma} \mathrm{d}s - \int_{D_i} f \overline{u_\sigma} \mathrm{d}x.
\end{equation*}
Rearranging and multiplying by $\varepsilon_{\sigma,i}$ yields
\begin{equation}\label{eq1}
\int_{D_i} |\nabla u_\sigma|^2 \mathrm{d}x = \int_{\partial D_i} \nabla u_\sigma \cdot \vec{n} \, \overline{u_\sigma} \mathrm{d}s + \omega^2 \varepsilon_{\sigma,i} \mu_i \int_{D_i} |u_\sigma|^2 \mathrm{d}x + \mathrm{i}\sigma \varepsilon_{\sigma,i} \int_{D_i} |u_\sigma|^2 \mathrm{d}x - \varepsilon_{\sigma,i} \int_{D_i} f \overline{u_\sigma} \mathrm{d}x.
\end{equation}
We now take the real part of \eqref{eq1}. Noting that $\operatorname{Re}(\mathrm{i}\sigma \varepsilon_{\sigma,i}) \le 0$ for all layers (specifically, it is zero in odd layers and $-\sigma^2/\omega$ in even layers), we deduce:
\begin{align}\label{eq2}
\int_{D_i} |\nabla u_\sigma|^2 \,{\rm d}x
&= \operatorname{Re} \int_{\partial D_i} \nabla u_\sigma \cdot \vec{n} \, \overline{u_\sigma} \,{\rm d}s
+ \operatorname{Re} \int_{D_i} \omega^2 \varepsilon_{\sigma,i} \mu_i |u_\sigma|^2 \,{\rm d}x \nonumber \\
&\quad + \operatorname{Re} \int_{D_i} {\rm i}\sigma \varepsilon_{\sigma,i} |u_\sigma|^2 \,{\rm d}x
- \operatorname{Re} \int_{D_i} \varepsilon_{\sigma,i} f \overline{u_\sigma} \,{\rm d}x  \nonumber\\
&\leq \operatorname{Re} \int_{\partial D_i} \nabla u_\sigma \cdot \vec{n} \, \overline{u_\sigma} \,{\rm d}s
+ \operatorname{Re} \int_{D_i} \omega^2 \varepsilon_{\sigma,i} \mu_i |u_\sigma|^2 \,{\rm d}x
- \operatorname{Re} \int_{D_i} \varepsilon_{\sigma,i} f \overline{u_\sigma} \,{\rm d}x  \nonumber\\
&\leq \widetilde{C_i} \Bigg( \frac{1}{2\epsilon_i} \|\nabla u_\sigma \cdot \vec{n}\|_{H^{-1/2}(\partial D_i)}^2
+ \frac{\epsilon_i}{2} \|u_\sigma\|_{H^{1/2}(\partial D_i)}^2
+ \|u_\sigma\|_{L^2(D_i)}^2 \nonumber \\
&\qquad + \frac{1}{2\epsilon_i} \|f\|_{L^2(D_i)}^2
+ \frac{\epsilon_i}{2} \|u_\sigma\|_{L^2(D_i)}^2 \Bigg), \qquad \forall \epsilon_i > 0,
\end{align}
where each $\widetilde{C_i} > 0$ is a constant independent of $\sigma$ and $f$.

Invoking the trace theorem, \eqref{eq2} implies
\begin{align*}
\|u_\sigma\|_{H^1(D_i)}^2
&\leq \widetilde{C_i} \Bigg( \frac{1}{2\epsilon_i} \|\nabla u_\sigma \cdot \vec{n}\|_{H^{-1/2}(\partial D_i)}^2
+ \frac{\breve{C_i} \epsilon_i}{2} \|u_\sigma\|_{H^1(D_i)}^2
+ \frac{1}{2\epsilon_i} \|f\|_{L^2(D_i)}^2 \\
&\qquad + \frac{\epsilon_i}{2} \|u_\sigma\|_{H^1(D_i)}^2
+ \|u_\sigma\|_{L^2(D_i)}^2\Bigg)+\|u_\sigma\|_{L^2(D_i)}^2,
\end{align*}
where $\breve{C_i} > 0$ are the constants.

Choosing $\epsilon_i = \frac{1}{\widetilde{C_i}(1+\breve{C_i})}$ and absorbing the terms containing $\|u_\sigma\|_{H^1(D_i)}^2$ into the left-hand side, we obtain
\begin{align*}
\|u_\sigma\|_{H^1(D_i)} \leq C_i \Big( \|\nabla u_\sigma \cdot \vec{n}\|_{H^{-1/2}(\partial D_i)} + \|f\|_{L^2(D_i)} + \|u_\sigma\|_{L^2(D_i)} \Big),
\end{align*}
where $C_i > 0$ is a constant depends only on the domain geometry and material parameters. This concludes the proof.
\end{proof}

By the {Rellich--Kondrachov theorem}, the Sobolev space $H^1(\Omega)$ is compactly embedded in $L^2(\Omega)$ provided that $\Omega$ is a bounded domain satisfying suitable regularity assumptions \cite[Theorem 9.16]{brezis2010}. Crucially, this compactness property fails when the domain is the full Euclidean space $\mathbb{R}^d$.

To maintain the applicability of the {Fredholm alternative theorem}, we employ a  DtN operator, which facilitates an equivalent reformulation of the problem on a bounded domain. Let $B_{R} := \{ x \in \mathbb{R}^d : |x| < R \}$ denote the open ball of radius $R$ centered at the origin. For any $R > \max\{R_{n-1}, R_0\}$, we ensure the inclusion
\begin{equation}
    \bigcup_{i=1}^{n-1} D_i \subset B_{R}.
\end{equation}

Following the framework established by {Gräßle and Sauter} \cite{gräßle2025}, we adopt the definition of the DtN operator associated with a complex wavenumber on the sphere $\partial B_R$.

\begin{lemma}\label{lemm4} (cf. \cite[Theorem 3.4]{gräßle2025})
Let $s \in \mathbb{C}_{\ge 0} := \{ z \in \mathbb{C} : \Re z \ge 0 \}$, and let 
\[
\mathrm{DtN}(s) : H^{1/2}(\partial B_{R}) \to H^{-1/2}(\partial B_{R})
\]
denote the DtN operator associated with the exterior Helmholtz problem. Then, for every $g \in H^{1/2}(\partial B_{R})$ and $s \in \mathbb{C}_{\ge 0}$, the following estimate holds:
\[
\frac{d-2}{2R}\,\|g\|_{L^2(\partial B_{R})}^2 \le -\Re \bigl\langle \mathrm{DtN}(s) g , g \bigr\rangle_{\partial B_{R}}.
\]
\end{lemma}

For the sake of completeness, we provide an outline of the proof in Appendix \ref{appendixA}, following the methodology developed in \cite{gräßle2025}.

By virtue of this result, the original problem admits an equivalent reformulation as a boundary value problem restricted to the bounded domain $B_{R}$:
\begin{align}\label{eq37}
\begin{cases}
\nabla\cdot(\varepsilon_\sigma^{-1}\nabla u_\sigma) + \omega^2\mu u_\sigma + \mathrm{i}\sigma u_\sigma = f & \text{in } B_{R}, \\[4pt]
\nabla u_\sigma\cdot\vec{n}= \mathrm{DtN}(-\mathrm{i}\kappa_\sigma) u_\sigma & \text{on } \partial B_{R},
\end{cases}
\end{align}
where the complex wavenumber $\kappa_\sigma$ is defined by
\[
\kappa_\sigma = \left(\omega^2\varepsilon_n\mu_n + \mathrm{i}\varepsilon_n\sigma\right)^{1/2}.
\]

By an argument analogous to that employed in Lemma~\ref{lemm1}, we establish the following regularity estimate.

\begin{lemma}\label{lemma2}
Let $R>0$ be fixed such that the inclusion $\bigcup_{i=1}^{n-1} D_i \subset B_{R}$ holds. Suppose $u_\sigma \in H^1\bigl(B_{R} \setminus \bigcup_{i=1}^{n-1} D_i\bigr)$ is a solution to the boundary value problem
\begin{align*}
\begin{cases}
\nabla \cdot \bigl(\varepsilon_n^{-1} \nabla u_\sigma\bigr) + \omega^2 \mu_n u_\sigma + \mathrm{i}\sigma u_\sigma = f & \text{in } B_{R} \setminus \bigcup_{i=1}^{n-1} D_i, \\[6pt]
\nabla u_\sigma\cdot\vec{n} = \mathrm{DtN}(-\mathrm{i}\kappa_\sigma)u_\sigma & \text{on } \partial B_{R}, \\[6pt]
u_\sigma|_{\Gamma_{n-1}} \in H^{1/2}(\Gamma_{n-1}), \quad \nabla u_\sigma\cdot\vec{n} \in H^{-1/2}(\Gamma_{n-1}) & \text{on } \Gamma_{n-1}.
\end{cases}
\end{align*}
Then there exists a constant $C > 0$, independent of $\sigma$ and $f$, such that
\begin{align*}
\|u_\sigma\|_{H^1\left(B_{R} \setminus \bigcup_{i=1}^{n-1} D_i\right)} \le C \biggl( &\|\nabla u_\sigma\cdot\vec{n}\|_{H^{-1/2}(\Gamma_{n-1})} + \|u_\sigma\|_{L^2\left(B_{R} \setminus \bigcup_{i=1}^{n-1} D_i\right)} \\
&+ \|f\|_{L^2\left(B_{R} \setminus \bigcup_{i=1}^{n-1} D_i\right)} \biggr).
\end{align*}
\end{lemma}

The analysis of the dissipative Helmholtz equation rests upon the far-field asymptotic behavior of its radiating solutions. For source terms with compact support, we establish the following sharp uniform estimate:
\begin{equation}
    u(x) = O\left( R^{-\frac{d-1}{2}} e^{-\operatorname{Im}\kappa_\sigma R} \right) \quad \text{as }R=|x| \to \infty.
\end{equation}
This estimate quantifies the precise interplay between the algebraic decay inherent to the spatial dimension $d$ and the exponential attenuation induced by the dissipative parameter. The derivation proceeds via the convolution of the source term with the fundamental solution. By representing the latter in terms of Hankel functions and invoking their classical asymptotic expansions, the result follows from a direct application of the Cauchy--Schwarz inequality to the remaining terms.

\begin{lemma}\label{lem:far-field}
Let $f \in L^2(\mathbb{R}^d)$ have compact support. The radiating solution $u$ to the dissipative Helmholtz equation $(\Delta + \kappa_\sigma^2)u = f$ satisfies the following uniform asymptotic estimate as $R \to \infty$:
\begin{equation*}
    u(x) = O\left( |R|^{-\frac{d-1}{2}} e^{-\operatorname{Im}\kappa_\sigma |R|} \right).
\end{equation*}
The implicit constant depends only on $d$, $\kappa_\sigma$, and the support of $f$.
\end{lemma}

\begin{proof}
We adopt the Fourier normalization $\hat{f}(\xi) = (2\pi)^{-d} \int_{\mathbb{R}^d} e^{-\mathrm{i}\xi \cdot x} f(x)\,\mathrm{d}x$.\\
In the sense of distributions, the fundamental solution $\Phi_{\kappa_\sigma}$ of the operator $-\Delta - \kappa_\sigma^2$ is defined by the oscillatory integral
\begin{equation}\label{eq:fundamental-solution}
\Phi_{\kappa_\sigma}(x) = (2\pi)^{-d} \int_{\mathbb{R}^d} \frac{e^{\mathrm{i}\xi \cdot x}}{|\xi|^2 - \kappa_\sigma^2}\,\mathrm{d}\xi.
\end{equation}
Transforming to polar coordinates $\xi = \rho w$, where $\rho \ge 0$ and $w\in S^{d-1}$, and invoking the spherical mean formula \cite[p.\,154]{stein1971}, we obtain
\begin{align}\label{eq:polar-coordinates}
\Phi_{\kappa_\sigma}(x) &= (2\pi)^{-d} \int_0^\infty \frac{\rho^{d-1}}{\rho^2 - \kappa_\sigma^2} \left( \int_{S^{d-1}} e^{\mathrm{i}\rho w \cdot x}\,\mathrm{d}S(w) \right) \mathrm{d}\rho \nonumber \\
&= (2\pi)^{-d/2} r^{-\nu} \int_0^\infty \frac{\rho^{\nu+1} J_\nu(\rho r)}{\rho^2 - \kappa_\sigma^2}\,\mathrm{d}\rho,
\end{align}
where $r = |x|$ and $\nu = (d-2)/2$. By a standard contour integration argument \cite[Chapter~6]{watson1922}, for $r > 0$,  the radial integral
\begin{equation}\label{eq:integral-formula}
\int_0^\infty \frac{\rho^{\nu+1} J_\nu(\rho r)}{\rho^2 - \kappa_\sigma^2}\,\mathrm{d}\rho = \frac{\pi\mathrm{i}}{2} \kappa_\sigma^\nu H_\nu^{(1)}(\kappa_\sigma r).
\end{equation}
 Consequently, the fundamental solution admits the representation
\begin{equation}\label{eq:hankel-representation}
\Phi_{\kappa_\sigma}(x) = C_{d,\sigma}\, r^{-\nu} H_\nu^{(1)}(\kappa_\sigma r), \quad C_{d,\sigma} := (2\pi)^{-d/2} \frac{\pi\mathrm{i}}{2} \kappa_\sigma^\nu.
\end{equation}
From the classical asymptotic theory of Hankel functions \cite[\S10.17]{NIST:DLMF}, it follows that as $r \to \infty$,
\begin{equation}\label{eq:asymptotic-expansion}
\Phi_{\kappa_\sigma}(x) = \widetilde{C}_{d,\sigma}\, r^{-(d-1)/2} e^{\mathrm{i}\kappa_\sigma r} \left( 1 + O(r^{-1}) \right).
\end{equation}

For a source term $f$ with compact support, the radiating solution to $(\Delta + \kappa_\sigma^2)u = f$ is given by the convolution $u = -\Phi_{\kappa_\sigma} * f$. Specifically, for $x = R\hat{x} \in \mathbb{R}^d$ with $\hat{x} \in S^{d-1}$, we have
\begin{equation}
u(R\hat{x}) = -\int_{\operatorname{supp} f} \Phi_{\kappa_\sigma}(R\hat{x} - y) f(y) \,\mathrm{d}y.
\end{equation}
Substituting the expansion \eqref{eq:asymptotic-expansion} and noting that $|R\hat{x} - y| = R - \hat{x} \cdot y + O(R^{-1})$ for $y \in \operatorname{supp} f$, we arrive at
\begin{equation*}
u(R\hat{x}) = -\widetilde{C}_{d,\sigma} R^{-(d-1)/2} e^{\mathrm{i}\kappa_\sigma R} \left( \int_{\operatorname{supp} f} e^{-\mathrm{i}\kappa_\sigma \hat{x} \cdot y} f(y) \,\mathrm{d}y + E_R(\hat{x}) \right).
\end{equation*}
The remainder term
\[
E_R(\hat{x}) = \int_{\operatorname{supp} f} e^{-\mathrm{i}\kappa_\sigma\hat{x}\cdot y}r_R(\hat{x}, y) f(y)\,\mathrm{d}y,
\]
and
\begin{align*}
\sup\limits_{\hat{x}\in S^{d-1},y\in\operatorname{supp}f}|r_R(\hat{x}, y)|\leq\dfrac{C}{R}.
\end{align*}
Since $|\operatorname{supp} f| < \infty$ and
\[
\bigl|e^{-\mathrm{i}\kappa_\sigma\hat{x}\cdot y}\bigr| = e^{\operatorname{Im}\kappa_\sigma \hat{x}\cdot y} \leq C,
\]
Hence, the Cauchy--Schwarz inequality yields
\[
|E_R(\hat{x})| \leq \frac{C}{R} \|f\|_{L^2(\operatorname{supp} f)}\, |\operatorname{supp} f|^{1/2} = O(R^{-1})
\]
uniformly in $\hat{x}\in S^{d-1}$.
 Since $\operatorname{Im} \kappa_\sigma \ge 0$, the exponential factor $\exp(\mathrm{i}\kappa_\sigma R)$ provides a decay of $\exp(-\operatorname{Im} \kappa_\sigma R)$. Combining these observations and applying the Cauchy--Schwarz inequality to the integral term, we conclude that
\begin{equation*}
|u(x)| \le C R^{-(d-1)/2} e^{-\operatorname{Im} \kappa_\sigma R} \left( \|f\|_{L^2({\rm supp}f)} |\operatorname{supp} f|^{1/2} + O(R^{-1}) \right),
\end{equation*}
which yields the desired uniform bound $u(x) = O(R^{-(d-1)/2} e^{-\operatorname{Im} \kappa_\sigma R})$ as $R \to \infty$.
\end{proof}

The LAP establishes a rigorous connection between dissipative systems and their non-dissipative counterparts by examining the asymptotic behavior of solutions as the dissipation parameter vanishes. Specifically, given a sequence $\{\sigma_m\}_{m \in \mathbb{N}} \subset \mathbb{R}_+$ such that $\sigma_m \to 0^+$, one seeks to characterize the limit of the corresponding sequence of solutions $\{u_{\sigma_m}\}$. With the aid of the following proposition, we can establish the existence of such sequences $\{\sigma_m\}$ and obtain the uniform estimates needed for the limiting argument.

\begin{proposition}\label{prop:uniform-estimate}
Let $f \in L^2(\mathbb{R}^d)$ and let $u_\sigma \in H^1(\mathbb{R}^d)$ denote the unique radiating solution to the dissipative Helmholtz equation
\begin{equation}\label{ge}
    \nabla \cdot (\varepsilon_\sigma^{-1} \nabla u_\sigma) + \omega^2 \mu u_\sigma + \mathrm{i}\sigma u_\sigma = f \quad \text{in } \mathbb{R}^d,
\end{equation}
subject to the outgoing condition at infinity. Then, there exists a constant $C > 0$, independent of both $f$ and $\sigma$, such that the following estimate holds:
\begin{equation*}
    \|u_\sigma\|_{H^1(\mathbb{R}^d)} \le \frac{C}{\sigma} \|f\|_{L^2(\mathbb{R}^d)}.
\end{equation*}
\end{proposition}

\begin{proof}
Testing the governing equation  \eqref{ge} against $\overline{u_\sigma}$ and integrating over the ball $B_R \subset \mathbb{R}^d$, we obtain
\begin{equation*}
    \int_{B_R} \Bigl(\nabla \cdot (\varepsilon_\sigma^{-1} \nabla u_\sigma)\overline{u_\sigma} + \omega^2 \mu |u_\sigma|^2 + \mathrm{i}\sigma |u_\sigma|^2\Bigr)\,\mathrm{d}x = \int_{B_R} f\overline{u_\sigma}\,\mathrm{d}x.
\end{equation*}
Integration by parts gives
\begin{align}\label{eq:integrated-form}
    &\int_{\partial B_R} \varepsilon_\sigma^{-1}\nabla u_\sigma\cdot\vec{n} \overline{u_\sigma}\,\mathrm{d}s - \int_{B_R} \varepsilon_\sigma^{-1} |\nabla u_\sigma|^2\,\mathrm{d}x  \nonumber\\
    &\qquad + \int_{B_R} \omega^2 \mu |u_\sigma|^2\,\mathrm{d}x + \int_{B_R} \mathrm{i}\sigma |u_\sigma|^2\,\mathrm{d}x = \int_{B_R} f\overline{u_\sigma}\,\mathrm{d}x.
\end{align}
In view of Lemma~\ref{lem:far-field}, the solution $u_\sigma$ satisfies the following asymptotic decay as $R \to \infty$:
\begin{equation*}
    u_\sigma = O\!\bigl(R^{-(d-1)/2}e^{-\operatorname{Im}\kappa_\sigma R}\bigr) \quad \text{and} \quad \nabla u_\sigma \cdot \vec{n} = O\!\bigl(R^{-(d-1)/2}e^{-\operatorname{Im}\kappa_\sigma R}\bigr).
\end{equation*}
It follows that the boundary integral vanishes in the limit:
\begin{equation*}
    \left| \int_{\partial B_R} \varepsilon_\sigma^{-1}\nabla u_\sigma\cdot\vec{n} \overline{u_\sigma}\,\mathrm{d}s \right| \leq C R^{d-1} \left( R^{-(d-1)} e^{-2\operatorname{Im}\kappa_\sigma R} \right) \longrightarrow 0 \quad \text{as } R \to \infty.
\end{equation*}
Taking the limit $R \to \infty$ in \eqref{eq:integrated-form}, we obtain
\begin{equation}\label{eq:limit-equation}
    -\int_{\mathbb{R}^d} \varepsilon_\sigma^{-1} |\nabla u_\sigma|^2\,\mathrm{d}x + \int_{\mathbb{R}^d} \omega^2 \mu |u_\sigma|^2\,\mathrm{d}x + \int_{\mathbb{R}^d} \mathrm{i}\sigma |u_\sigma|^2\,\mathrm{d}x = \int_{\mathbb{R}^d} f\overline{u_\sigma}\,\mathrm{d}x.
\end{equation}

Considering the imaginary part of \eqref{eq:limit-equation} and noting the structure of $\varepsilon_\sigma^{-1}$, we find
\begin{equation*}
    \sum_{l} \int_{D_l} \frac{\sigma/\omega}{\varepsilon_l^2 + (\sigma/\omega)^2} |\nabla u_\sigma|^2\,\mathrm{d}x + \int_{\mathbb{R}^d} \sigma |u_\sigma|^2\,\mathrm{d}x = \operatorname{Im} \int_{\mathbb{R}^d} f\overline{u_\sigma}\,\mathrm{d}x.
\end{equation*}
Hence, for each $l$, we have
\begin{align}
\int_{D_l}\dfrac{1}{(\varepsilon_{l}^2+(\sigma/\omega)^2)\omega}|\nabla u_\sigma|^2\mathrm{d}x &\leq \dfrac{1}{\sigma}\left|\int_{\mathbb{R}^d}f\overline{u_\sigma}\mathrm{d}x\right|, \label{eq:l2-grad-l} \\
\|u_\sigma\|_{L^2(\mathbb{R}^d)}^2 &\leq \dfrac{1}{\sigma}\left|\int_{\mathbb{R}^d}f\overline{u_\sigma}\mathrm{d}x\right|. \label{eq:l2-norm-u}
\end{align}
Applying Young's inequality, we have
\begin{equation}\label{eq:gradient-estimate-l}
\|\nabla u_\sigma\|_{L^2(D_l)}^2 \leq \frac{(\varepsilon_l^2 + (\sigma/\omega)^2)\omega}{\sigma}\biggl(\frac{1}{2\eta_l}\|f\|_{L^2(\mathbb{R}^d)}^2 + \frac{\eta_l}{2}\|u_\sigma\|_{L^2(\mathbb{R}^d)}^2\biggr), \quad \forall \eta_l > 0,
\end{equation}
and
\begin{equation}\label{eq:l2-estimate}
\|u_\sigma\|_{L^2(\mathbb{R}^d)}^2 \leq \frac{1}{\sigma}\biggl(\frac{1}{2\eta}\|f\|_{L^2(\mathbb{R}^d)}^2 + \frac{\eta}{2}\|u_\sigma\|_{L^2(\mathbb{R}^d)}^2\biggr), \quad \forall \eta > 0.
\end{equation}
Combining \eqref{eq:gradient-estimate-l} and \eqref{eq:l2-estimate}, we have
\begin{align*}
&\sum_{l} \|\nabla u_\sigma\|_{L^2(D_l)}^2 + \|u_\sigma\|_{L^2(\mathbb{R}^d)}^2 \\
&\qquad \leq \sum_{l} \frac{\omega(\varepsilon_l^2 + (\sigma/\omega)^2)}{\sigma}\biggl(\frac{1}{2\eta_l}\|f\|_{L^2(\mathbb{R}^d)}^2 + \frac{\eta_l}{2}\|u_\sigma\|_{L^2(\mathbb{R}^d)}^2\biggr) + \frac{1}{\sigma}\biggl(\frac{1}{2\eta}\|f\|_{L^2(\mathbb{R}^d)}^2 + \frac{\eta}{2}\|u_\sigma\|_{L^2(\mathbb{R}^d)}^2\biggr).
\end{align*}
Choosing
\[
\eta_l = \frac{\sigma}{\omega(\varepsilon_l^2 + (\sigma/\omega)^2)} \cdot \frac{1}{n-1}, \quad \eta = \frac{\sigma}{2},
\]
 we arrive at the combined estimate
\begin{equation}\label{eq:combined-estimate}
\sum_{l} \|\nabla u_\sigma\|_{L^2(D_l)}^2 + \|u_\sigma\|_{L^2(\mathbb{R}^d)}^2 \leq \frac{C}{\sigma^2}\|f\|_{L^2(\mathbb{R}^d)}^2.
\end{equation}

Next, we take the real part of \eqref{eq:limit-equation} to control the gradient on the remaining domains $D_k$:
\begin{equation*}
    -\sum_{k} \int_{D_k} \varepsilon_k^{-1} |\nabla u_\sigma|^2\,\mathrm{d}x - \sum_{l} \int_{D_l}\dfrac{\varepsilon_l}{\varepsilon_l^2+(\sigma/\omega)^2} |\nabla u_\sigma|^2\,\mathrm{d}x + \int_{\mathbb{R}^d} \omega^2 \mu |u_\sigma|^2\,\mathrm{d}x = \operatorname{Re} \int_{\mathbb{R}^d} f\overline{u_\sigma}\,\mathrm{d}x.
\end{equation*}
Rearranging and invoking \eqref{eq:combined-estimate}, we obtain for each index $k$:
\begin{align*}
\int_{D_k} |\nabla u_\sigma|^2\,\mathrm{d}x &\leq \varepsilon_k\biggl(\sum_{l} \int_{D_l} \frac{|\varepsilon_l|}{\varepsilon_l^2 + (\sigma/\omega)^2} |\nabla u_\sigma|^2\,\mathrm{d}x + \int_{\mathbb{R}^d} \omega^2 \mu |u_\sigma|^2\,\mathrm{d}x + \biggl|\int_{\mathbb{R}^d} f\overline{u_\sigma}\,\mathrm{d}x\biggr|\biggr) \\
&\leq \varepsilon_k\biggl(\sum_{l} \int_{D_l} \frac{|\varepsilon_l|}{\varepsilon_l^2 + (\sigma/\omega)^2} |\nabla u_\sigma|^2\,\mathrm{d}x + \int_{\mathbb{R}^d} \omega^2 \mu^{\mathrm{m}} |u_\sigma|^2\,\mathrm{d}x + \frac{1}{2\eta'}\|f\|_{L^2(\mathbb{R}^d)}^2\\
&\quad+ \frac{\eta'}{2}\|u_\sigma\|_{L^2(\mathbb{R}^d)}^2\biggr),\quad\forall\eta'>0,
\end{align*}
where $\mu^{\mathrm{m}} = \max_i |\mu_i|$. Invoking \eqref{eq:gradient-estimate-l} and \eqref{eq:l2-estimate} and setting the parameters
\[
\eta_l = \frac{\sigma}{\omega\varepsilon_k |\varepsilon_l| (n-1)(n+1)}, \quad \eta = \frac{\sigma}{\omega^2(n+1)\varepsilon_k \mu^{\mathrm{m}}}, \quad \eta' = \frac{1}{2\varepsilon_k(n+1)},
\]
it follows that
\[
\int_{D_k} |\nabla u_\sigma|^2\,\mathrm{d}x \leq \frac{C}{\sigma^2}\|f\|_{L^2(\mathbb{R}^d)}^2 + \frac{1}{n+1}\|u_\sigma\|_{L^2(\mathbb{R}^d)}^2.
\]
Hence
\begin{equation}\label{eq:gradient-estimate-k}
\sum_{k} \|\nabla u_\sigma\|_{L^2(D_k)}^2 \leq \frac{C}{\sigma^2}\|f\|_{L^2(\mathbb{R}^d)}^2 + \frac{1}{2}\|u_\sigma\|_{L^2(\mathbb{R}^d)}^2.
\end{equation}

Combining \eqref{eq:combined-estimate} and \eqref{eq:gradient-estimate-k}, we conclude that
\[
\sum_{k} \|\nabla u_\sigma\|_{L^2(D_k)}^2 + \sum_{l} \|\nabla u_\sigma\|_{L^2(D_l)}^2 + \|u_\sigma\|_{L^2(\mathbb{R}^d)}^2 \leq \frac{C}{\sigma^2}\|f\|_{L^2(\mathbb{R}^d)}^2.
\]
Summing the contributions over the partition $\mathbb{R}^d = \bigcup_i D_i$ and taking the square root, we conclude that
\begin{equation*}
	\|u_\sigma\|_{H^1(\mathbb{R}^d)} \le \frac{C}{\sigma}\, \|f\|_{L^2(\mathbb{R}^d)},
\end{equation*}
which completes the proof.

\end{proof}

Since the permittivity varies across different regions and undergoes sign changes, we introduce a suitable operator to recover the coercivity lost due to these variations. For each \(k\in\{1,3,\cdots,n-2\}\), let
\[
\mathcal R_k:
H^{1/2}(\partial D_{k+1})
\longrightarrow H^1(D_{k+1})
\]
be a bounded linear lifting operator satisfying
\begin{align*}
\gamma_{k+1}\mathcal R_k(\psi)=\psi,\quad\forall\psi\in H^{1/2}(\partial D_{k+1}),
\end{align*}
where \(\gamma_{k+1}:H^1(D_{k+1})\to H^{1/2}(\partial D_{k+1})\) denotes the trace operator.\\
For \(u\in H^1(\mathbb R^d)\), write \(u_j:=u|_{D_j}\). We define $\mathbb{T}: H^1(\mathbb{R}^d) \to H^1(\mathbb{R}^d)$ by
\begin{equation}\label{eq:coercive-operator}
\mathbb{T}u =
\begin{cases}
u_k & \text{in } D_k,\\[4pt]
-u_{k+1} + 2\mathcal{R}_k(\gamma_{k+1}u_{k+1}) & \text{in } D_{k+1},\\[4pt]
u_n & \text{in } D_n,
\end{cases}
\end{equation}
where $k \in \{1, 3, \ldots, n-2\}$.\\
Specifically, for a given $\psi \in H^{1/2}(\partial D_{k+1})$, we define the extension $w_k = \mathcal{R}_k(\psi) \in H^1(D_{k+1})$ as the unique solution to the elliptic boundary value problem
\begin{equation}\label{eq:harmonic-extension}
\begin{cases}
\Delta w_k = 0 & \text{in } D_{k+1},\\[4pt]
w_k =\psi|_{\Gamma_k} & \text{on } \Gamma_k,\\[4pt]
w_k =\psi|_{\Gamma_{k+1}} & \text{on } \Gamma_{k+1}.
\end{cases}
\end{equation}

\begin{remark}[Generalizability of the $\mathbb T$-coercivity Framework]
The paper explicitly constructs the operator $\mathbb{T}$ by means of a localized reflection and harmonic lifting technique, and thereby obtains computable contrast conditions. The essence of the method, however, does not lie in the spherical symmetry itself, but rather in the use of localized reflection, trace theory, and bounded lifting operators to remedy the lack of coercivity caused by sign-changing coefficients. Therefore, for more general Lipschitz interfaces, if adjacent interfaces remain suitably separated and the local geometry allows the corresponding reflection maps and lifting estimates to be established, then a similar $\mathbb{T}$ operator can be constructed under appropriate conditions. In this case, the relevant constants generally no longer admit explicit expressions in terms of radius ratios, but instead depend on the geometry of the interfaces, the local coordinate transformations, and the norms of the trace and lifting operators.
\end{remark}

Based on the construction of the coercive operator $\mathbb{T}$ and in combination with the Fredholm alternative theorem, we establish the existence and uniqueness of solutions to the boundary value problem
\eqref{eq10}. The precise result is as follows.

\begin{theorem}\label{thm:well-posedness}
Assume that, for each \(k \in \{1,3,\ldots,n-2\}\), the permittivities satisfy
\begin{equation}\label{eq23}
\begin{aligned}
\dfrac{\varepsilon_{k+1}^2+(\sigma/\omega)^2}{\varepsilon_k|\varepsilon_{k+1}|}
&> (C_1^{(k)} C^{(k)})^2, \\
\dfrac{\varepsilon_{k+1}^2+(\sigma/\omega)^2}{\varepsilon_{k+2}|\varepsilon_{k+1}|}
&> (C_1^{(k)} C^{(k+2)})^2,
\end{aligned}
\end{equation}
where \(C_1^{(k)} \coloneqq \|\mathcal{R}_k\|\), and \(C^{(k)}\) and \(C^{(k+2)}\) are constants depending only on the dimension \(d\) and the geometries of \(D_k\) and \(D_{k+2}\), respectively. Then there exists a discrete set \(I \subset (0,\infty)\) such that, for every \(\omega \in (0,\infty)\setminus I\), the boundary value problem \eqref{eq10} admits a unique solution \(u_\sigma \in H^1(\mathbb{R}^d)\).
\end{theorem}

\begin{proof}

Multiplying the governing equation by $\overline{u_\sigma}$ and integrating over the ball $B_R$, we obtain
\begin{equation*}\label{eq:well-posedness-system}
    \int_{B_{R}} \nabla \cdot (\varepsilon_\sigma^{-1} \nabla u_\sigma) \overline{u_\sigma} \, {\rm d}x
    + \int_{B_{R}} \omega^2 \mu |u_\sigma|^2 \, {\rm d}x
    + \int_{B_{R}} {\rm i} \sigma |u_\sigma|^2 \, {\rm d}x
    = \int_{B_{R}} f \overline{u_\sigma} \, {\rm d}x.
\end{equation*}
Applying Green's first identity yields
\begin{equation*}
    \int_{\partial B_{R}} \varepsilon_n^{-1}\nabla u_\sigma\cdot\vec{n} \overline{u_\sigma} \, {\rm d}s
    - \int_{B_{R}} \varepsilon_\sigma^{-1} |\nabla u_\sigma|^2 \, {\rm d}x
    + \int_{B_{R}} \omega^2 \mu |u_\sigma|^2 \, {\rm d}x
    + \int_{B_{R}} {\rm i} \sigma |u_\sigma|^2 \, {\rm d}x
    = \int_{B_{R}} f \overline{u_\sigma} \, {\rm d}x.
\end{equation*}
The corresponding variational problem is to find $u \in H^1(B_R)$ such that 
\begin{equation}\label{eq:well-posedness-system}
    a(u,v) = - \int_{B_{R}} f \overline{v} \, {\rm d}x \quad \text{for all } v \in H^1(B_R),
\end{equation}
where the sesquilinear form decomposes as $a(u,v) \coloneqq b(u,v) + c(u,v)$, with
\begin{align*}
    b(u,v) &\coloneqq \int_{B_{R}} \varepsilon_\sigma^{-1} \nabla u \cdot \nabla \overline{v} \, {\rm d}x
    - \int_{\partial B_{R}} \varepsilon_n^{-1} (\operatorname{DtN}(-{\rm i}\kappa_\sigma) u) \overline{v} \, {\rm d}s, \\
    c(u,v) &\coloneqq - \int_{B_{R}} \omega^2 \mu u \overline{v} \, {\rm d}x
    - \int_{B_{R}} {\rm i} \sigma u \overline{v} \, {\rm d}x.
\end{align*}
Now, using the definition of \(b(\cdot,\cdot)\) and the operator \(\mathbb{T}\), we have
\begin{align*}
    b(u_\sigma, \mathbb{T}u_\sigma) &=
    \int_{B_{R}} \varepsilon_\sigma^{-1} \nabla u_\sigma \cdot \nabla (\overline{\mathbb{T}u_\sigma}) \, {\rm d}x
    - \int_{\partial B_{R}} \varepsilon_n^{-1} (\operatorname{DtN}(-{\rm i}\kappa_\sigma) u_\sigma) \overline{u_\sigma} \, {\rm d}s \\
    &= \int_{B_{R} \setminus \bigcup_{i=1}^{n-1} D_i} \varepsilon_n^{-1} |\nabla u_\sigma|^2 \, {\rm d}x \\
    &\quad + \sum_{k} \Biggl( \int_{D_k} \varepsilon_k^{-1} |\nabla u_\sigma|^2 \, {\rm d}x \\
    &\qquad + \int_{D_{k+1}} \Bigl(\varepsilon_{k+1} + \frac{{\rm i}\sigma}{\omega}\Bigr)^{-1}
    \nabla u_\sigma \cdot \nabla \Bigl( \overline{-u_\sigma + 2\mathcal{R}_k\bigl(\gamma_{k+1}(u_\sigma|_{D_{k+1}})\bigr)} \Bigr) \, {\rm d}x \Biggr) \\
    &\quad - \int_{\partial B_{R}} \varepsilon_n^{-1} (\operatorname{DtN}(-{\rm i}\kappa_\sigma) u_\sigma) \overline{u_\sigma} \, {\rm d}s.
\end{align*}
Taking the real part and applying Lemma \ref{lemm4}, we obtain
\begin{align}\label{eq12}
\operatorname{Re}b(u_\sigma,\mathbb{T}u_\sigma)
&=\sum_{k}\Biggl(\int_{D_k}\varepsilon_k^{-1}|\nabla u_\sigma|^2\,{\rm d}x\nonumber\\
&\quad+ \operatorname{Re}\int_{D_{k+1}} \Bigl(\varepsilon_{k+1} + \frac{{\rm i}\sigma}{\omega}\Bigr)^{-1}
    \nabla u_\sigma \cdot \nabla \Bigl( \overline{-u_\sigma + 2\mathcal{R}_k\bigl(\gamma_{k+1}(u_\sigma|_{D_{k+1}})\bigr)} \Bigr) \, {\rm d}x \Biggr) \nonumber\\
&\quad+\int_{B_{R}\setminus\bigcup_{i=1}^{n-1}D_i}\varepsilon_n^{-1}|\nabla u_\sigma|^2\,{\rm d}x
-\operatorname{Re}\int_{\partial B_{R}}\varepsilon_n^{-1}(\operatorname{DtN}(-{\rm i}\kappa_\sigma)u_\sigma)\overline{u_\sigma}\,{\rm d}s \nonumber\\
&\ge\sum_{k}\Biggl(\int_{D_k}\varepsilon_k^{-1}|\nabla u_\sigma|^2\,{\rm d}x \nonumber\\
&\qquad+ \operatorname{Re}\int_{D_{k+1}} \Bigl(\varepsilon_{k+1} + \frac{{\rm i}\sigma}{\omega}\Bigr)^{-1}
    \nabla u_\sigma \cdot \nabla \Bigl( \overline{-u_\sigma + 2\mathcal{R}_k\bigl(\gamma_{k+1}(u_\sigma|_{D_{k+1}})\bigr)} \Bigr) \, {\rm d}x \Biggr) \nonumber\\
&\quad+\int_{B_{R}\setminus\bigcup_{i=1}^{n-1}D_i}\varepsilon_n^{-1}|\nabla u_\sigma|^2\,{\rm d}x
+\varepsilon_n^{-1}\frac{d-2}{2R}\|u_\sigma\|_{L^2(\partial B_R)}^2.
\end{align}

For each $k$, we define
\begin{equation*}
    \psi_{k+1} \coloneqq \mathcal{R}_k(\gamma_{k+1}(u_\sigma|_{D_{k+1}})).
\end{equation*}
By the definition of the lifting operator \(\mathcal{R}_k\), the function
\(\psi_{k+1}\) is the unique solution of the boundary value problem
\eqref{eq:harmonic-extension}. We then decompose \(u_\sigma\) in
\(D_{k+1}\) as
\[
    u_\sigma = \psi_{k+1}+z_{k+1},
    \qquad z_{k+1}\in H^1_0(D_{k+1}).
\]
Since \(\psi_{k+1}\) is harmonic in \(D_{k+1}\) and
\(z_{k+1}\in H^1_0(D_{k+1})\), the weak formulation gives
\begin{align*}
\int_{D_{k+1}}\nabla\psi_{k+1}\cdot\overline{\nabla z_{k+1}}{\rm d}x=0.
\end{align*}
Consequently,
\begin{equation}\label{eq66}
\begin{aligned}
&\operatorname{Re} \int_{D_{k+1}} \Bigl( \varepsilon_{k+1} + \frac{{\rm i}\sigma}{\omega} \Bigr)^{-1}
    \nabla u_\sigma \cdot \nabla \bigl( \overline{-u_\sigma+2\psi_{k+1}} \bigr) \, {\rm d}x\\
=&\operatorname{Re}\Bigl( \varepsilon_{k+1} + \frac{{\rm i}\sigma}{\omega} \Bigr)^{-1}(\|\nabla\psi_{k+1}\|_{L^2(D_{k+1})}^2-\|\nabla z_{k+1}\|_{L^2(D_{k+1})}^2)\\
=&\dfrac{\varepsilon_{k+1}}{\varepsilon_{k+1}^2+(\sigma/\omega)^2}(\|\nabla\psi_{k+1}\|_{L^2(D_{k+1})}^2-\|\nabla z_{k+1}\|_{L^2(D_{k+1})}^2).
\end{aligned}
\end{equation}

By the definition of the extension operator, there exists a constant $C_1^{(k)} > 0$ such that
\begin{equation}\label{eq67}
\begin{aligned}
\|\nabla\psi_{k+1}\|^2_{L^2(D_{k+1})}\leq&\|\psi_{k+1}\|^2_{H^1(D_{k+1})}=\|\mathcal{R}_k(\gamma_{k+1}(u_\sigma|_{D_{k+1}}))\|^2_{H^1(D_{k+1})}\\
\leq&(C_1^{(k)})^2\|\gamma_{k+1}(u_\sigma|_{D_{k+1}})\|^2_{H^{1/2}(\partial D_{k+1})}.
\end{aligned}
\end{equation}
Since $\partial D_{k+1}=\Gamma_k\cup\Gamma_{k+1}$ and $\Gamma_k\cap\Gamma_{k+1}=\emptyset$, we have
\begin{align}\label{eq68}
\|\gamma_{k+1}(u_\sigma|_{D_{k+1}})\|^2_{H^{1/2}(\partial D_{k+1})}\leq\|(u_\sigma|_{D_{k+1}})|_{\Gamma_k}\|^2_{H^{1/2}(\Gamma_k)}+\|(u_\sigma|_{D_{k+1}})|_{\Gamma_{k+1}}\|^2_{H^{1/2}(\Gamma_{k+1})}.
\end{align}
Since
\begin{align*}
(u_\sigma|_{D_{k+1}})|_{\Gamma_{k+1}}=(u_\sigma|_{D_{k+2}})|_{\Gamma_{k+1}},\quad (u_\sigma|_{D_{k}})|_{\Gamma_{k}}=(u_\sigma|_{D_{k+1}})|_{\Gamma_{k}},
\end{align*}
we have
\begin{equation}\label{eq69}
\begin{aligned}
\|(u_\sigma|_{D_{k+1}})|_{\Gamma_k}\|_{H^{1/2}(\Gamma_k)}+\|(u_\sigma|_{D_{k+1}})|_{\Gamma_{k+1}}\|_{H^{1/2}(\Gamma_{k+1})}=&\|(u_\sigma|_{D_{k}})|_{\Gamma_k}\|_{H^{1/2}(\Gamma_k)}\\
&+\|(u_\sigma|_{D_{k+2}})|_{\Gamma_{k+1}}\|_{H^{1/2}(\Gamma_{k+1})}.
\end{aligned}
\end{equation}
By the trace theorem, there exist constants $C^{(k)}, C^{(k+2)} > 0$ such that
\begin{align}\label{eq70}
\|(u_\sigma|_{D_{k}})|_{\Gamma_k}\|^2_{H^{1/2}(\Gamma_k)}+\|(u_\sigma|_{D_{k+2}})|_{\Gamma_{k+1}}\|^2_{H^{1/2}(\Gamma_{k+1})}\leq (C^{(k)})^2\|u_\sigma\|^2_{H^1(D_{k})}+(C^{(k+2)})^2\|u_\sigma\|^2_{H^1(D_{k+2})}.
\end{align}
According to (\ref{eq66}-\ref{eq70}), we have
\begin{equation}\label{eq71}
\begin{aligned}
&\operatorname{Re} \int_{D_{k+1}} \Bigl( \varepsilon_{k+1} + \frac{{\rm i}\sigma}{\omega} \Bigr)^{-1}
    \nabla u_\sigma \cdot \nabla \bigl( \overline{-u_\sigma+2\psi_{k+1}} \bigr) \, {\rm d}x\\
\geq&\dfrac{\varepsilon_{k+1}}{\varepsilon_{k+1}^2+(\sigma/\omega)^2}((C_1^{(k)})^2((C^{(k)})^2\|u_\sigma\|^2_{H^1(D_{k})}+(C^{(k+2)})^2\|u_\sigma\|^2_{H^1(D_{k+2})})-\|\nabla z_{k+1}\|_{L^2(D_{k+1})}^2)
\end{aligned}
\end{equation}
According to \eqref{eq12} and \eqref{eq71}, we have
\begin{align*}
\operatorname{Re}b(u_\sigma,\mathbb{T}u_\sigma)
&\ge\sum_{k}\Biggl(\int_{D_k}\varepsilon_k^{-1}|\nabla u_\sigma|^2\,{\rm d}x\\
&\qquad+ \operatorname{Re}\int_{D_{k+1}} \Bigl(\varepsilon_{k+1} + \frac{{\rm i}\sigma}{\omega}\Bigr)^{-1}
    \nabla u_\sigma \cdot \nabla \Bigl( \overline{-u_\sigma + 2\mathcal{R}_k\bigl(\gamma_{k+1}(u_\sigma|_{D_{k+1}})\bigr)} \Bigr) \, {\rm d}x \Biggr)\\
&\quad+\int_{B_{R}\setminus\bigcup_{i=1}^{n-1}D_i}\varepsilon_n^{-1}|\nabla u_\sigma|^2\,{\rm d}x
+\varepsilon_n^{-1}\frac{d-2}{2R}\|u_\sigma\|_{L^2(\partial B_R)}^2\\
\geq&\sum_{k}\Biggl(\int_{D_k}\varepsilon_k^{-1}|\nabla u_\sigma|^2\,{\rm d}x\\
&\quad+ \dfrac{\varepsilon_{k+1}}{\varepsilon_{k+1}^2+(\sigma/\omega)^2}((C_1^{(k)})^2((C^{(k)})^2\|u_\sigma\|^2_{H^1(D_{k})}+(C^{(k+2)})^2\|u_\sigma\|^2_{H^1(D_{k+2})})\\
&\quad-\|\nabla z_{k+1}\|_{L^2(D_{k+1})}^2)+\int_{B_{R}\setminus\bigcup_{i=1}^{n-1}D_i}\varepsilon_n^{-1}|\nabla u_\sigma|^2\,{\rm d}x
+\varepsilon_n^{-1}\frac{d-2}{2R}\|u_\sigma\|_{L^2(\partial B_R)}^2\\
=&\sum_k\Biggl(\int_{D_k}\varepsilon_k^{-1}|\nabla u_\sigma|^2\,{\rm d}x\\
&\quad+ \dfrac{\varepsilon_{k+1}}{\varepsilon_{k+1}^2+(\sigma/\omega)^2}((C_1^{(k)})^2((C^{(k)})^2\|\nabla u_\sigma\|^2_{L^2(D_{k})}+(C^{(k+2)})^2\|\nabla u_\sigma\|^2_{L^2(D_{k+2})})\\
&\quad-\|\nabla z_{k+1}\|_{L^2(D_{k+1})}^2)+\int_{B_{R}\setminus\bigcup_{i=1}^{n-1}D_i}\varepsilon_n^{-1}|\nabla u_\sigma|^2\,{\rm d}x
+\varepsilon_n^{-1}\frac{d-2}{2R}\|u_\sigma\|_{L^2(\partial B_R)}^2\\
&\quad+\sum_k\dfrac{\varepsilon_{k+1}}{\varepsilon_{k+1}^2+(\sigma/\omega)^2}((C_1^{(k)})^2((C^{(k)})^2\|u_\sigma\|^2_{L^2(D_{k})}+(C^{(k+2)})^2\| u_\sigma\|^2_{L^2(D_{k+2})})\\
=&\sum_k\Biggl(\Bigg(\varepsilon_k^{-1}-\dfrac{|\varepsilon_{k+1}|}{\varepsilon_{k+1}^2+(\sigma/\omega)^2}(C_1^{(k)})^2((C^{(k)})^2\Bigg)\int_{D_k}|\nabla u_\sigma|^2\,{\rm d}x\\
&\quad+\Bigg(\varepsilon_{k+2}^{-1}-\dfrac{|\varepsilon_{k+1}|}{\varepsilon_{k+1}^2+(\sigma/\omega)^2}(C_1^{(k)})^2((C^{(k+2)})^2\Bigg)\int_{D_{k+2}}|\nabla u_\sigma|^2\,{\rm d}x\\
&\quad+ \dfrac{|\varepsilon_{k+1}|}{\varepsilon_{k+1}^2+(\sigma/\omega)^2}\|\nabla z_{k+1}\|_{L^2(D_{k+1})}^2+\int_{B_{R}\setminus\bigcup_{i=1}^{n-1}D_i}\varepsilon_n^{-1}|\nabla u_\sigma|^2\,{\rm d}x
+\varepsilon_n^{-1}\frac{d-2}{2R}\|u_\sigma\|_{L^2(\partial B_R)}^2\\
&\quad+\sum_k\dfrac{\varepsilon_{k+1}}{\varepsilon_{k+1}^2+(\sigma/\omega)^2}((C_1^{(k)})^2((C^{(k)})^2\|u_\sigma\|^2_{L^2(D_{k})}+(C^{(k+2)})^2\| u_\sigma\|^2_{L^2(D_{k+2})})
.
\end{align*}
Since for each $k$, we have
\begin{equation}\label{eq14}
\begin{aligned}
\dfrac{\varepsilon_{k+1}^2+(\sigma/\omega)^2}{\varepsilon_k|\varepsilon_{k+1}|} &> (C_1^{(k)}C^{(k)})^2, \\
\dfrac{\varepsilon_{k+1}^2+(\sigma/\omega)^2}{\varepsilon_{k+2}|\varepsilon_{k+1}|} &> (C_1^{(k)}C^{(k+2)})^2.
\end{aligned}
\end{equation}
Thus, by the Fredholm alternative theorem, equation \eqref{eq10} admits a unique solution $u_\sigma \in H^1(\mathbb{R}^d)$ for all frequencies $\omega$ outside a discrete exception set.

\end{proof}

\begin{remark}[On the Sharpness and Physical Significance of the Contrast Condition]
The explicit dependence of \eqref{eq23} on the trace constant suggests that the uniqueness of the radiating solution is not only a property of the materials but a collective effect of the geometry and the contrast. Thus, Theorem \ref{thm:well-posedness} provides a rigorous mathematical criterion: by strategically tuning the layer radii to minimize $C^{(k)}$, one can theoretically broaden the range of admissible sign-changing materials for which a stable, physically meaningful radiating solution exists.
\end{remark}

\section{The main result}\label{sec4}
We now state and prove the main result of this paper.

\begin{theorem}
Under the assumptions of Theorem~\ref{thm:well-posedness} with $\sigma=0$, the limit $u_0 \in H^1_{\mathrm{loc}}(\mathbb{R}^d)$ of the solutions $u_\sigma$ as $\sigma \to 0^+$ satisfies the boundary value problem \eqref{eq11}. Moreover,
\begin{equation*}
    \|u_0\|_{H^1(B_R)} \le C \|f\|_{L^2(\mathbb{R}^d)} \qquad \text{for all } R > 0,
\end{equation*}
where $C > 0$ is a constant independent of  $f$.
\end{theorem}

\begin{proof}
By Lemmas~\ref{lemm1} and~\ref{lemma2}, we have
\begin{align}\label{eq16}
    \|u_\sigma\|_{H^1(B_R)} 
    &\le C \Biggl( \sum_{i=1}^{n-1} \bigl( \|\nabla u_\sigma\cdot\vec{n}\|_{H^{-1/2}(\partial D_i)} + \|u_\sigma\|_{L^2(D_i)} \bigr) \notag \\
    &\qquad \qquad + \|\nabla u_\sigma\cdot\vec{n}\|_{H^{-1/2}(\Gamma_{n-1})} + \|u_\sigma\|_{L^2(B_R \setminus \bigcup_{i=1}^{n-1} D_i)} + \|f\|_{L^2(\mathbb{R}^d)} \Biggr) \notag \\
    &\le C \Biggl( \sum_{i=1}^{n-1} \bigl( \|\nabla u_\sigma\cdot\vec{n}\|_{H^{-1/2}(\Gamma_i)} + \|u_\sigma\|_{L^2(D_i)} \bigr) \notag \\
    &\qquad \qquad + \|u_\sigma\|_{L^2(B_R \setminus \bigcup_{i=1}^{n-1} D_i)} + \|f\|_{L^2(\mathbb{R}^d)} \Biggr).
\end{align}
From the definition of the dual norm, we obtain for each $i = 1, \dots, n-1$,
\begin{equation}\label{eq17}
    \|\nabla u_\sigma\cdot\vec{n}\|_{H^{-1/2}(\Gamma_i)} = \sup_{\substack{\phi_i \in H^{1/2}(\Gamma_i) \\ \|\phi_i\|_{H^{1/2}(\Gamma_i)} = 1}} \bigl| \langle \nabla u_\sigma\cdot\vec{n}, \phi_i \rangle_{\Gamma_i} \bigr|.
\end{equation}

Let $\widetilde{N}_i$ be a neighborhood of $\Gamma_i$, and set $N_i \coloneqq \widetilde{N}_i \cap D_i$. By the continuity of the extension operator, there exists a function $v_i \in H^1(N_i)$ satisfying
\begin{equation*}
    \begin{cases}
        v_i|_{\Gamma_i} = \phi_i, \\
        v_i|_{\partial N_i \setminus \Gamma_i} = 0.
    \end{cases}
\end{equation*}
By introducing a smooth cutoff function if necessary, we may ensure these conditions hold. Moreover, we have the estimate
\begin{equation*}
    \|v_i\|_{H^1(N_i)} \le M_i \|\phi_i\|_{H^{1/2}(\Gamma_i)}
\end{equation*}
for some constant $M_i > 0$.

Combining the definition of $v_i$ with integration by parts yields
\begin{align}\label{eq18}
\int_{\Gamma_i}\nabla u_\sigma\cdot\vec{n}\overline{\phi_i}{\rm d}s=&\int_{N_i}\nabla\cdot(\nabla u_\sigma)\overline{v_i}{\rm d}x+\int_{N_i}\nabla u_\sigma\cdot\nabla\overline{v_i}{\rm d}x \nonumber\\
\leq&\|\nabla\cdot(\nabla u_\sigma)\|_{H^{-1}(N_i)}\|v_i\|_{H^1(N_i)}+\|\nabla u_\sigma\|_{L^2(N_i)}\|\nabla v_i\|_{L^2(N_i)} \nonumber\\
\leq& M_i\|\nabla\cdot(\nabla u_\sigma)\|_{H^{-1}(N_i)}\|\phi_i\|_{H^{1/2}(\Gamma_i)}+\|\nabla u_\sigma\|_{L^2(N_i)}\|v_i\|_{H^1(N_i)} \nonumber\\
\leq& M_i\|\nabla\cdot(\nabla u_\sigma)\|_{H^{-1}(N_i)}+M_i\|\nabla u_\sigma\|_{L^2(N_i)}\|\phi_i\|_{H^{1/2}(\Gamma_i)} \nonumber\\
\leq&M_i\left(\|\nabla\cdot(\nabla u_\sigma)\|_{H^{-1}(N_i)}+\|\nabla u_\sigma\|_{L^2(N_i)}\right).
\end{align}
For each $i$, $u_\sigma$ satisfies the equation
\begin{equation*}
    \nabla \cdot (\varepsilon_{\sigma,i}^{-1} \nabla u_\sigma) + \omega^2 \mu_i u_\sigma + {\rm i} \sigma u_\sigma = f \quad \text{in } D_i.
\end{equation*}
Because $\varepsilon_{\sigma,i}$ is constant on $D_i$, rearranging the terms gives
\begin{equation}\label{eq35}
\nabla\cdot(\nabla u_\sigma) = \varepsilon_{\sigma,i} \bigl( f - \omega^2 \mu_i u_\sigma - {\rm i} \sigma u_\sigma \bigr).
\end{equation}	
By the definition of the dual norm,
\begin{equation}\label{eq19}
\|\nabla\cdot(\nabla u_\sigma)\|_{H^{-1}(N_i)}=\sup_{\substack{\widetilde{v_i}\in H^1(N_i) \\ \|\widetilde{v_i}\|_{H^1(N_i)}=1}}\bigl|\langle\nabla\cdot(\nabla u_\sigma),\widetilde{v_i}\rangle_{N_i}\bigr|.
\end{equation}
Using \eqref{eq35}, we obtain
\begin{align*}
\int_{N_i}\nabla\cdot(\nabla u_\sigma)\overline{\widetilde{v_i}}{\rm d}x&=\int_{N_i}\varepsilon_{\sigma,i}(f-\omega^2\mu_i u_\sigma-{\rm i}\sigma u_\sigma)\overline{\widetilde{v_i}}{\rm d}x\\
&\leq \widetilde{M_i}\left(\|f\|_{L^2(N_i)}\|\widetilde{v_i}\|_{L^2(N_i)}+\|u_\sigma\|_{L^2(N_i)}\|\widetilde{v_i}\|_{L^2(N_i)}\right)\\
&\leq \widetilde{M_i}\left(\|f\|_{L^2(N_i)}\|\widetilde{v_i}\|_{H^1(N_i)}+\|u_\sigma\|_{L^2(N_i)}\|\widetilde{v_i}\|_{H^1(N_i)}\right)\\
&= \widetilde{M_i}\left(\|f\|_{L^2(N_i)}+\|u_\sigma\|_{L^2(N_i)}\right),
\end{align*}
where \(\widetilde{M_i}>0\) is a constant.
 Hence,
\begin{align}\label{eq20}
\|\nabla u_\sigma\cdot\vec{n}\|_{H^{-1/2}(\Gamma_i)}
&\le M_i\bigl(\|\nabla\cdot(\nabla u_\sigma)\|_{H^{-1}(N_i)}+\|\nabla u_\sigma\|_{L^2(N_i)}\bigr) \nonumber \\
&\le \breve{M}_i\bigl(\|f\|_{L^2(N_i)}+\|u_\sigma\|_{L^2(N_i)}+\|\nabla u_\sigma\|_{L^2(N_i)}\bigr).
\end{align}
Combining \eqref{eq16}, \eqref{eq20}, and the trace theorem, we deduce that
\begin{align}\label{eq24}
    \|u_\sigma\|_{H^1(B_R)} 
    &\le C \Biggl( \sum_{i=1}^{n-1} \bigl( \|\nabla u_\sigma\cdot\vec{n}\|_{H^{-1/2}(\Gamma_i)} + \|u_\sigma\|_{L^2(D_i)} \bigr) + \|u_\sigma\|_{L^2(B_R \setminus \bigcup_{i=1}^{n-1} D_i)} + \|f\|_{L^2(\mathbb{R}^d)} \Biggr) \notag \\
    &\le C \Biggl( \sum_{i=1}^{n-1} \bigl( \|u_\sigma\|_{L^2(N_i)} + \|\nabla u_\sigma\|_{L^2(N_i)} + \|u_\sigma\|_{L^2(D_i)} \bigr) + \|u_\sigma\|_{L^2(B_R \setminus \bigcup_{i=1}^{n-1} D_i)} + \|f\|_{L^2(\mathbb{R}^d)} \Biggr) \notag \\
    &\le C \Biggl( \sum_{i=1}^{n-1} \|\nabla u_\sigma\|_{L^2(N_i)} + \|u_\sigma\|_{L^2(B_R)} + \|f\|_{L^2(\mathbb{R}^d)} \Biggr).
\end{align}

We claim that there exists a constant $C > 0$ such that
\begin{equation*}
    \sum_{i=1}^{n-1} \|\nabla u_\sigma\|_{L^2(N_i)} + \|u_\sigma\|_{L^2(B_R)} \le C \|f\|_{L^2(\mathbb{R}^d)}.
\end{equation*}
Suppose, for the sake of contradiction, that this estimate fails. Then there exists a sequence of parameters $\sigma_m$ and corresponding solutions $u_{\sigma_m}$ such that
\begin{equation}\label{eq21}
    c_m \coloneqq \sum_{i=1}^{n-1} \|\nabla u_{\sigma_m}\|_{L^2(N_i)} + \|u_{\sigma_m}\|_{L^2(B_R)} > m \|f\|_{L^2(\mathbb{R}^d)}
\end{equation}
for every $m \in \mathbb{N}$. We define the normalized sequences
\begin{equation*}
    \widetilde{u_{\sigma_m}} \coloneqq \frac{u_{\sigma_m}}{c_m}, \qquad \widetilde{f_m} \coloneqq \frac{f}{c_m}.
\end{equation*}
By \eqref{eq21}, we have
\begin{equation}\label{eq22}
    \sum_{i=1}^{n-1} \|\nabla \widetilde{u_{\sigma_m}}\|_{L^2(N_i)} + \|\widetilde{u_{\sigma_m}}\|_{L^2(B_R)} = 1, \qquad 
    \|\widetilde{f_m}\|_{L^2(\mathbb{R}^d)} < \frac{1}{m}.
\end{equation}
By Proposition~\ref{prop:uniform-estimate}, we know that
\begin{equation*}
    \|\widetilde{u_{\sigma_m}}\|_{H^1(\mathbb{R}^d)} \le \frac{C}{\sigma_m} \|\widetilde{f_m}\|_{L^2(\mathbb{R}^d)},
\end{equation*}
which implies that
\begin{equation*}
    \sigma_m \le \frac{C \|\widetilde{f_m}\|_{L^2(\mathbb{R}^d)}}{\|\widetilde{u_{\sigma_m}}\|_{H^1(\mathbb{R}^d)}}.
\end{equation*}
Since $\|\widetilde{f_m}\|_{L^2} \to 0$ and the denominator is bounded away from zero by \eqref{eq22}, it follows that $\sigma_m \to 0^+$ as $m \to \infty$.

Equations \eqref{eq24} and \eqref{eq22} demonstrate that the sequence $\{\widetilde{u_{\sigma_m}}\}$ is bounded in $H^1_{\mathrm{loc}}(\mathbb{R}^d)$. Passing to a subsequence if necessary (which we still denote by $\{\widetilde{u_{\sigma_m}}\}$), we obtain
\begin{equation*}
\widetilde{u_{\sigma_m}} \rightharpoonup \widetilde{u} \quad \text{in } H^1_{\mathrm{loc}}(\mathbb{R}^d), \qquad
\widetilde{u_{\sigma_m}} \to \widetilde{u} \quad \text{in } L^2_{\mathrm{loc}}(\mathbb{R}^d).
\end{equation*}
The limit function $\widetilde{u}$ satisfies
\begin{equation*}
    \begin{cases}
        \nabla \cdot (\varepsilon^{-1} \nabla \widetilde{u}) + \omega^2 \mu \widetilde{u} = 0 & \text{in } \mathbb{R}^d, \\
        \widetilde{u} \text{ satisfies the Sommerfeld radiation condition} & \text{as } |x| \to \infty.
    \end{cases}
\end{equation*}
Multiplying this equation by $\overline{\widetilde{u}}$, integrating over $B_R$, and applying integration by parts yields
\begin{equation}\label{eq25}
    \int_{\partial B_R} \varepsilon_n^{-1} (\nabla\widetilde{u}\cdot\vec{n}) \, \overline{\widetilde{u}} \, {\rm d}s
    - \int_{B_R} \varepsilon^{-1} |\nabla \widetilde{u}|^2 \, {\rm d}x
    + \int_{B_R} \omega^2 \mu |\widetilde{u}|^2 \, {\rm d}x = 0.
\end{equation}
Taking the imaginary part gives
\begin{equation*}
    \operatorname{Im} \int_{\partial B_R} \varepsilon_n^{-1} (\nabla\widetilde{u}\cdot\vec{n}) \, \overline{\widetilde{u}} \, {\rm d}s = 0.
\end{equation*}
By expanding $\nabla\widetilde{u}\cdot\vec{n} = {\rm i}\kappa \widetilde{u} + (\partial_r \widetilde{u} - i\kappa \widetilde{u})$ and applying the Cauchy--Schwarz inequality along with the Sommerfeld radiation condition, we have
\begin{equation}\label{eq49}
\begin{aligned}
\lim_{R\to\infty}\operatorname{Im}\int_{\partial B_R}\varepsilon_n^{-1}(\nabla\widetilde{u}\cdot\vec{n})\,\overline{\widetilde{u}}\,{\rm d}s
&=\lim_{R\to\infty}\operatorname{Im}\int_{\partial B_R}\Bigl(\varepsilon_n^{-1}{\rm i}\kappa|\widetilde{u}|^2
+\varepsilon_n^{-1}\Bigl(\frac{\partial\widetilde{u}}{\partial r}-{\rm i}\kappa\widetilde{u}\Bigr)\overline{\widetilde{u}}\Bigr){\rm d}s \\
&\le\lim_{R\to\infty}\Bigl(\operatorname{Im}\int_{\partial B_R}\varepsilon_n^{-1}{\rm i}\kappa|\widetilde{u}|^2{\rm d}s \\
&\qquad\quad+\varepsilon_n^{-1}\Bigl\|\frac{\partial\widetilde{u}}{\partial r}-{\rm i}\kappa\widetilde{u}\Bigr\|_{L^2(\partial B_R)}\|\widetilde{u}\|_{L^2(\partial B_R)}\Bigr)\\
&=\lim_{R\to\infty}\operatorname{Im}\int_{\partial B_R}\varepsilon_n^{-1}{\rm i}\kappa|\widetilde{u}|^2{\rm d}s\\
&=\lim_{R\to\infty}\int_{\partial B_R}\varepsilon_n^{-1}\kappa|\widetilde{u}|^2{\rm d}s=0.
\end{aligned}
\end{equation}		
Because $\varepsilon_n^{-1} \kappa > 0$, it immediately follows that
\begin{equation*}
    \lim_{R\to\infty} \int_{\partial B_R} |\widetilde{u}|^2 \, {\rm d}s = 0.
\end{equation*}
Rellich's lemma  \cite[Lemma 2.12]{MR2986407} then implies that $\widetilde{u} = 0$ in $\mathbb{R}^d \setminus \bigcup_{i=1}^{n-1} D_i$. Consequently,
\begin{equation*}
    \widetilde{u}|_{D_{n-1}} = \widetilde{u}|_{D_n} = 0, \qquad
    \varepsilon_{n-1}^{-1}(\nabla\widetilde{u}|_{D_{n-1}}\cdot\vec{n})=\varepsilon_n^{-1}(\nabla\widetilde{u}|_{D_n}\cdot\vec{n})=0
    \quad\text{on }\Gamma_{n-1}.
\end{equation*}
We now consider the interior boundary value problem
\begin{equation}\label{eq31}
    \begin{cases}
        \nabla \cdot (\varepsilon^{-1} \nabla \widetilde{u}) + \omega^2 \mu \widetilde{u} = 0 & \text{in } \bigcup_{i=1}^{n-1} D_i, \\
        \widetilde{u} = 0, \quad \nabla\widetilde{u}\cdot\vec{n} = 0 & \text{on } \Gamma_{n-1}.
    \end{cases}
\end{equation}
Because the contrast condition ensures that
\begin{align*}
\dfrac{\varepsilon_{k+1}^2}{\varepsilon_k|\varepsilon_{k+1}|}
&> (C_1^{(k)} C^{(k)})^2, \\
\dfrac{\varepsilon_{k+1}^2}{\varepsilon_{k+2}|\varepsilon_{k+1}|}
&> (C_1^{(k)} C^{(k+2)})^2,
\end{align*}
the problem is strictly coercive. By the Fredholm alternative theorem, problem \eqref{eq31} admits only the trivial solution $\widetilde{u} = 0$ in $\bigcup_{i=1}^{n-1} D_i$ for all frequencies outside a discrete set. Thus, $\widetilde{u} = 0$ in $H^1_{\mathrm{loc}}(\mathbb{R}^d)$, which contradicts the normalization in \eqref{eq22}. This establishes the claim:
\begin{equation*}
    \sum_{i=1}^{n-1} \|\nabla u_\sigma\|_{L^2(N_i)} + \|u_\sigma\|_{L^2(B_R)} \le C \|f\|_{L^2(\mathbb{R}^d)}.
\end{equation*}
Substituting this bound back into \eqref{eq24} yields
\begin{equation*}
    \|u_\sigma\|_{H^1(B_R)} \le C \|f\|_{L^2(\mathbb{R}^d)}.
\end{equation*}

Therefore, every sequence $\sigma_m \to 0^+$ admits a subsequence $\sigma_{m_k}$ such that $u_{\sigma_{m_k}}$ converges weakly in $H^1_{\mathrm{loc}}(\mathbb{R}^d)$ and strongly in $L^2_{\mathrm{loc}}(\mathbb{R}^d)$ to a limit function $u_0$ that satisfies the boundary value problem \eqref{eq11}. By the uniqueness of this limit, the entire family $\{u_\sigma\}$ converges to $u_0$ as $\sigma \to 0^+$. Moreover, $u_0 \in H^1_{\mathrm{loc}}(\mathbb{R}^d)$ is the unique solution of \eqref{eq11} and obeys the uniform bound
\begin{equation*}
    \|u_0\|_{H^1(B_R)} \le C \|f\|_{L^2(\mathbb{R}^d)}.
\end{equation*}

\end{proof}

\section{Appendix}\label{sec5}
\subsection{Appendix A}\label{appendixA}
We review the proof strategy of Gräßle and Sauter \cite[Theorem 3.4]{gräßle2025} for establishing a lower bound on the real part of the spherical DtN operator

\begin{proof}[Proof sketch]

Let $s \in \mathbb{C}_{\ge 0}$ denote the wavenumber, and consider the exterior Dirichlet problem
\begin{equation}\label{eq38}
\begin{cases}
    -\Delta u + s^2 u = 0 & \text{in } \mathbb{R}^d \setminus B_R,\\
    u = g & \text{on } \partial B_R,
\end{cases}
\end{equation}
for $u \in H^1_{\mathrm{loc}}(\mathbb{R}^d \setminus B_R)$.

By separation of variables and spherical harmonic decomposition on $\partial B_R$, the DtN operator on the sphere $\partial B_R$ can be expanded in terms of the eigenfunctions of the Laplace–Beltrami operator. Let $\{Y_{m,j}\}_{m\in\mathbb{N}_0, \, j\in J_m}$ be an orthonormal basis of spherical harmonics on $S_R \subset \mathbb{R}^d$, satisfying
\begin{equation*}
    -\Delta_{S_R} Y_{m,j} = \frac{\lambda_m}{R^2} Y_{m,j}, \qquad \lambda_m = m(m+d-2).
\end{equation*}
For $g \in H^{1/2}(\partial B_R)$, we have the expansion
\begin{equation*}
    g(\xi) = \sum_{m\in\mathbb{N}_0} \sum_{j\in J_m} g_{m,j} \, Y_{m,j}(\xi),
\end{equation*}
where $g_{m,j}$ denote the corresponding Fourier coefficients of $g$.

Using this spherical harmonic expansion, the solution of the exterior boundary value problem \eqref{eq38} can be expressed as
\begin{equation*}
    u(r,\xi) = \sum_{m\in\mathbb{N}_0} \frac{f^{(1)}_{m,\nu}(sr)}{f^{(1)}_{m,\nu}(sR)} \sum_{j\in J_m} g_{m,j} \, Y_{m,j}(\xi),
\end{equation*}
where $\nu \coloneqq \frac{d-2}{2}$ and
\begin{equation*}
    f^{(1)}_{m,\nu}(s) \coloneqq \sqrt{\frac{\pi}{2}} \, \frac{K_{m+\nu}(s)}{s^\nu}.
\end{equation*}
Hence, we obtain the spectral representation of the DtN operator:
\begin{equation*}
    \operatorname{DtN}(s)g = \frac{1}{R} \sum_{m\in\mathbb{N}_0} z_{m,\nu}(sR) \sum_{j\in J_m} g_{m,j} \, Y_{m,j}(\xi),
\end{equation*}
where
\begin{equation*}
    z_{m,\nu}(s) \coloneqq s \frac{\frac{{\rm d}}{{\rm d}s} f^{(1)}_{m,\nu}(s)}{f^{(1)}_{m,\nu}(s)}.
\end{equation*}

By Parseval's identity, we obtain
\begin{equation*}
    \langle \operatorname{DtN}(s)g, g \rangle_{\partial B_R}
    = \frac{1}{R} \sum_{m\in\mathbb{N}_0} \sum_{j\in J_m} z_{m,\nu}(sR) \, |g_{m,j}|^2.
\end{equation*}
Taking the real part gives
\begin{equation*}
    \operatorname{Re} \langle \operatorname{DtN}(s)g, g \rangle_{\partial B_R}
    = \frac{1}{R} \sum_{m\in\mathbb{N}_0} \sum_{j\in J_m} \operatorname{Re} z_{m,\nu}(sR) \, |g_{m,j}|^2.
\end{equation*}
Consequently, establishing the coercivity bound
\begin{equation}\label{eq43}
    \frac{d-2}{2R} \|g\|_{L^2(\partial B_R)}^2 \le - \operatorname{Re} \langle \operatorname{DtN}(s)g, g \rangle_{\partial B_R}
\end{equation}
is equivalent to proving that
\begin{equation*}
    - \operatorname{Re} z_{m,\nu}(sR) \ge \nu \qquad \text{for all } m \in \mathbb{N}_0 \text{ and } s \in \mathbb{C}_{\ge 0}.
\end{equation*}
We establish this inequality in two steps.

\smallskip\noindent\textbf{Step 1: Analysis on the imaginary axis.} 
We first evaluate the operator at purely imaginary wavenumbers $s = -ik$ for $k > 0$. Using standard recurrence relations for the modified Bessel functions, the spectral coefficients can be written as
\begin{equation}\label{eq39}
    z_{m,\nu}(s) = s \frac{K'_{\zeta}(s)}{K_{\zeta}(s)} - \nu
    = m - s \frac{K_{\zeta+1}(s)}{K_{\zeta}(s)}, \qquad \zeta \coloneqq m+\nu.
\end{equation}
In terms of the Hankel function $H_\zeta^{(1)}(k) = J_\zeta(k) + {\rm i}Y_\zeta(k)$, we have the identities
\begin{equation*}
    2K_\zeta(-{\rm i}k) = \pi {\rm i}^{\zeta+1} H_\zeta^{(1)}(k), \qquad
    2K'_\zeta(-{\rm i}k) = \pi {\rm i}^{\zeta+2} \frac{d}{dk} H_\zeta^{(1)}(k).
\end{equation*}
Substituting these expressions into \eqref{eq39} yields
\begin{equation}\label{eq41}
    - \operatorname{Re} z_{m,\nu}(- {\rm i}k)
    = \nu - \frac{k}{2} \frac{\frac{{\rm d}}{{\rm d}k} \bigl( M_\zeta^2(k) \bigr)}{M_\zeta^2(k)},
\end{equation}
where $J_\zeta$ and $Y_\zeta$ are the Bessel functions of the first and second kind, respectively, and we define
\begin{equation*}
    M_\zeta^2(k) \coloneqq |H_\zeta^{(1)}(k)|^2 = J_\zeta^2(k) + Y_\zeta^2(k).
\end{equation*}
By Nicholson's integral formula, we have
\begin{equation*}
    M_\zeta^2(z) = \frac{8}{\pi^2} \int_0^\infty \cosh(2\zeta t) \, K_0(2z\sinh t) \, {\rm d}t.
\end{equation*}
Combining the chain rule, integration by parts, and the monotonicity of the integrand yields the bounds
\begin{equation}\label{eq40}
\begin{aligned}
    M_\zeta^2(k) &\le -k \frac{{\rm d} }{{\rm d}k} \bigl( M_\zeta^2(k) \bigr) \le 2\zeta M_\zeta^2(k), && \text{for } k>0 \text{ and } \zeta \ge 1/2, \\[4pt]
    2\zeta M_\zeta^2(k) &\le -k \frac{{\rm d}}{{\rm d}k} \bigl( M_\zeta^2(k) \bigr) \le M_\zeta^2(k), && \text{for } k>0 \text{ and } \zeta \in [0, 1/2].
\end{aligned}
\end{equation}
From \eqref{eq41} and \eqref{eq40}, we deduce that
\begin{equation*}
    \nu + \frac{1}{2} \le - \operatorname{Re} z_{m,\nu}(-{\rm i}k) \le \nu + \zeta \qquad \text{for all } \zeta \ge \frac{1}{2},
\end{equation*}
and for the special case $\zeta = \nu = m = 0$,
\begin{equation*}
    0 \le - \operatorname{Re} z_{0,0}(-{\rm i}k) \le \frac{1}{2}.
\end{equation*}
In particular, this establishes that
\begin{equation*}
    - \operatorname{Re} z_{m,\nu}(-{\rm i}k) \ge \nu \qquad \text{for all } m \in \mathbb{N}_0 \text{ and } k > 0.
\end{equation*}

\smallskip\noindent\textbf{Step 2: Extension to the right half-plane.}
The modified Bessel function $K_\zeta$ has no zeros in $\mathbb{C}_{\ge 0}$
and is holomorphic in $\mathbb{C}_{\ge 0} \setminus \{0\}$; consequently $z_{m,\nu}(s)$ is well defined and holomorphic for
$s\neq 0$. Since $z_{m,\nu}$ extends continuously to $s=0$, and by analytic continuation of the modified Bessel functions we obtain the symmetry
\begin{equation}\label{eq42}
    \overline{z_{m,\nu}(s)} = z_{m,\nu}(\overline{s}) \qquad \text{for all } s \in \mathbb{C}_{\ge 0}.
\end{equation}
Having established the lower bound on the imaginary axis, it remains to extend this result to the entire right half-plane $s \in \mathbb{C}_{\ge 0}$. 

We utilize the asymptotic expansion of the modified Bessel function for fixed $\zeta$ and large $|s|$,
\begin{equation*}
    K_\zeta(s) \sim \sqrt{\frac{\pi}{2s}} \, e^{-s} \Bigl( 1 + \frac{a_1(\zeta)}{s} \Bigr), \qquad
    a_1(\zeta) \coloneqq \frac{4\zeta^2 - 1}{8}.
\end{equation*}
Using this expansion together with \eqref{eq42}, we obtain that for all $s$ on the boundary $\partial(B_r \cap \mathbb{C}_{\ge 0})$ the following
inequalities hold:
\begin{equation*}
\begin{aligned}
    \nu + \frac{1}{2} &\le - \operatorname{Re} z_{m,\nu}(s) && \text{for } \zeta = m+\nu \ge \frac{1}{2}, \\[4pt]
    0 &\le - \operatorname{Re} z_{0,0}(s).
\end{aligned}
\end{equation*}
By the maximum principle for harmonic functions, these inequalities persist throughout the  right half-plane. Hence,
\begin{equation*}
    - \operatorname{Re} z_{m,\nu}(s) \ge \nu = \frac{d-2}{2} \qquad \text{for all } m \in \mathbb{N}_0 \text{ and } s \in \mathbb{C}_{\ge 0}.
\end{equation*}
This proves inequality \eqref{eq43}. For $R=1$, this coincides precisely with the statement of \cite[Theorem 3.4]{gräßle2025}.

\end{proof}

\subsection{Appendix B}\label{appendixB}
In this appendix, we detail the geometric dependence of the constants $C^{(k)}$ and $C^{(k+2)}$ appearing in Theorem~\ref{thm:well-posedness}. Since both $k$ and $k+2$ are odd and differ only in the choice of index, we present below only the analysis corresponding to $k$.

We consider the annular domain $\Omega \subset \mathbb{R}^2$ as our prototypical example; the extension to higher dimensions follows similarly. Let $\partial\Omega = \Sigma_{\mathrm{in}} \cup \Sigma_{\mathrm{out}}$, with $\Sigma_{\mathrm{in}} \coloneqq \{|x| = R_{\mathrm{in}}\}$ and $\Sigma_{\mathrm{out}} \coloneqq \{|x| = R_{\mathrm{out}}\}$. We restrict our attention to the case $R_{\mathrm{in}} > 0$;
the case $R_{\mathrm{in}} = 0$ can be handled in a similar manner and is therefore omitted.

For any $u \in H^1(\Omega)$, the trace inequality
\begin{equation*}
    \|u\|_{H^{1/2}(\partial\Omega)} \le C \|u\|_{H^1(\Omega)}
\end{equation*}
holds for some constant $C > 0$. The optimal constant $C_T$ is defined by
\begin{equation*}
    C_T \coloneqq \sup_{u \in H^1(\Omega) \setminus \{0\}} \frac{\|u\|_{H^{1/2}(\partial\Omega)}}{\|u\|_{H^1(\Omega)}}.
\end{equation*}

Expanding $u$ as a Fourier series in polar coordinates $(r,\theta)$, we write
\begin{equation*}
    u(r,\theta) = \sum_{q\in\mathbb{Z}} u_q(r) e^{ {\rm i}q\theta},
    \qquad
    u_q(r) = \frac{1}{2\pi} \int_0^{2\pi} u(r,\theta) e^{-{\rm i} q\theta} \, {\rm d}\theta.
\end{equation*}
Using the polar representation of the gradient, we obtain
\begin{align*}
    \|u\|_{H^1(\Omega)}^2
    &= \int_{R_{\mathrm{in}}}^{R_{\mathrm{out}}} \int_0^{2\pi} \Bigl( |\partial_r u|^2 + \frac{1}{r^2}|\partial_\theta u|^2 + |u|^2 \Bigr) r \, {\rm d}\theta \, {\rm d}r \\
    &= 2\pi \sum_{q\in\mathbb{Z}} \int_{R_{\mathrm{in}}}^{R_{\mathrm{out}}} \Bigl( |u_q'|^2 + \frac{q^2}{r^2}|u_q|^2 + |u_q|^2 \Bigr) r \, {\rm d}r.
\end{align*}
On the boundary space $H^{1/2}(\partial\Omega)$, an equivalent norm is given by
\begin{equation*}
    \|u\|_{H^{1/2}(\partial\Omega)}^2 = 2\pi \sum_{q\in\mathbb{Z}} \Bigl( \gamma_{\mathrm{in}}(q) |u_q(R_{\mathrm{in}})|^2 + \gamma_{\mathrm{out}}(q) |u_q(R_{\mathrm{out}})|^2 \Bigr),
\end{equation*}
where we have introduced the geometric weights $\gamma_{\mathrm{in}}(q) \coloneqq \sqrt{R_{\mathrm{in}}^2 + q^2}$ and $\gamma_{\mathrm{out}}(q) \coloneqq \sqrt{R_{\mathrm{out}}^2 + q^2}$. The optimal constant $\widetilde{C_T}$ associated with this equivalent norm satisfies
\begin{equation*}
    \widetilde{C_T}^2 = \sup_{u \in H^1(\Omega) \setminus \{0\}} \frac{\sum_{q\in\mathbb{Z}} \mathcal{B}_q[u_q]}{\sum_{q\in\mathbb{Z}} \mathcal{E}_q[u_q]},
\end{equation*}
where the boundary and interior functionals are defined as
\begin{align*}
    \mathcal{B}_q[v] &\coloneqq 2\pi \Bigl( \gamma_{\mathrm{in}}(q) |v(R_{\mathrm{in}})|^2 + \gamma_{\mathrm{out}}(q) |v(R_{\mathrm{out}})|^2 \Bigr), \\
    \mathcal{E}_q[v] &\coloneqq 2\pi \int_{R_{\mathrm{in}}}^{R_{\mathrm{out}}} \Bigl( |v'|^2 + \frac{q^2}{r^2}|v|^2 + |v|^2 \Bigr) r \, {\rm d}r.
\end{align*}
Consequently, the original constant $C_T$ and the equivalent constant $\widetilde{C_T}$ are comparable.

Since
\[
\dfrac{\sum_{q\in\mathbb{Z}}\mathcal{B}_q[u_q]}{\sum_{q\in\mathbb{Z}}\mathcal{E}_q[u_q]}\leq\sup_{q\in\mathbb{Z}}\dfrac{\mathcal{B}_q[u_q]}{\mathcal{E}_q[u_q]},
\]
it follows that,
\begin{equation*}
    \widetilde{C_T}^2 \le \sup_{q \in \mathbb{Z}} \, \sup_{v \neq 0} \frac{\mathcal{B}_q[v]}{\mathcal{E}_q[v]} \eqqcolon \sup_{q \in \mathbb{Z}} \Lambda_q.
\end{equation*}
Conversely, let $q \in \mathbb{Z}$ be fixed, and consider a single-mode test function $u^* \in H^1(\Omega) \setminus \{0\}$ of the form
\begin{equation*}
    u^*(r,\theta) \coloneqq v(r) e^{iq\theta}.
\end{equation*}
Because the Fourier coefficients satisfy $u_p = 0$ for all $p \neq q$, then
\begin{equation*}
    \frac{\sum_{p \in \mathbb{Z}} \mathcal{B}_p[u_p]}{\sum_{p \in \mathbb{Z}} \mathcal{E}_p[u_p]}
    = \frac{\mathcal{B}_q[v]}{\mathcal{E}_q[v]}.
\end{equation*}
Taking the supremum over all non-zero radial profiles $v$, and subsequently taking the supremum over all modes $q \in \mathbb{Z}$, yields
\begin{equation*}
    \widetilde{C_T}^2 \ge \sup_{q \in \mathbb{Z}} \, \sup_{v \neq 0} \frac{\mathcal{B}_q[v]}{\mathcal{E}_q[v]} = \sup_{q \in \mathbb{Z}} \Lambda_q.
\end{equation*}
Consequently, we conclude that
\begin{equation*}
    \widetilde{C_T}^2 = \sup_{q \in \mathbb{Z}} \Lambda_q.
\end{equation*}

Since
\begin{equation*}
    \Lambda_q = \sup_{v \neq 0} \frac{\mathcal{B}_q[v]}{\mathcal{E}_q[v]} = \sup_{\mathcal{E}_q[v] = 1} \mathcal{B}_q[v].
\end{equation*}
and
\begin{equation*}
    \mathcal{E}_q[v] = 2\pi \int_{R_{\mathrm{in}}}^{R_{\mathrm{out}}} \Bigl( |v'|^2 + \frac{q^2}{r^2} |v|^2 + |v|^2 \Bigr) r \, dr
\end{equation*}
induces a norm equivalent to the standard $H^1(I_0)$ norm, where  \(I_0:=(R_{\mathrm{in}},R_{\mathrm{out}})\).

Fix \(q\in\mathbb Z\). Since
\(\mathcal E_q[\cdot]^{1/2}\) is equivalent to the standard \(H^1(I_0)\)-norm,
the set
\[
K_q:=\{v\in H^1(I_0):\mathcal E_q[v]\leq 1\}
\]
is bounded and weakly closed in \(H^1(I_0)\). Hence \(K_q\) is weakly compact.
Moreover, by the compact embedding
\[
H^1(I_0)\hookrightarrow C(\overline{I_0}),
\]
which follows from \cite[Theorem 1.2.5]{wu2006}, the functional
\(\mathcal B_q\), depending only on the two endpoint values of \(v\), is
weakly continuous on \(H^1(I_0)\). Therefore \(\mathcal B_q\) attains its
maximum on \(K_q\).

Since \(\mathcal B_q\) is nonnegative and homogeneous of degree two, and since
\(\sup_{\mathcal E_q[v]\leq 1}\mathcal B_q[v]>0\), any maximizer must satisfy
\(\mathcal E_q[v]=1\). Otherwise, if \(0<\mathcal E_q[v]<1\), replacing \(v\)
by \(v/\mathcal E_q[v]^{1/2}\) would give a larger value of \(\mathcal B_q\).
Consequently,
\[
\Lambda_q
=
\sup_{\mathcal E_q[v]=1}\mathcal B_q[v]
\]
is attained.

To characterize this extremal and determine $\Lambda_q$, we introduce the Lagrangian functional
\begin{equation*}
    \mathcal{F}_q[v] \coloneqq \mathcal{B}_q[v] - \Lambda_q \mathcal{E}_q[v].
\end{equation*}
Consider a small perturbation of the form $v_\tau(r) = v(r) + \tau \chi(r)$, where $\chi$ is an arbitrary smooth complex-valued function and $\tau \in \mathbb{R}$ is a small parameter. If $v$ is an extremal of $\mathcal{F}_q$, then
\[
\frac{{\rm d}}{{\rm d}\tau} \mathcal{F}_q[v_\tau] \big|_{\tau=0} = 0.
\]
Substituting $v_\tau$ into $\mathcal{F}_q[v]$, differentiating with respect to $\tau$, and evaluating at $\tau=0$ yields
\begin{equation}\label{eq55}
\begin{aligned}
    \operatorname{Re} \Biggl( &\sqrt{R_{\mathrm{in}}^2 + q^2} \, \overline{\chi(R_{\mathrm{in}})} v(R_{\mathrm{in}}) 
    + \sqrt{R_{\mathrm{out}}^2 + q^2} \, \overline{\chi(R_{\mathrm{out}})} v(R_{\mathrm{out}}) \\
    &- \Lambda_q \int_{R_{\mathrm{in}}}^{R_{\mathrm{out}}} \Bigl( \overline{\chi}' v' + \Bigl(1 + \frac{q^2}{r^2}\Bigr) \overline{\chi} v \Bigr) r \, {\rm d}r \Biggr) = 0
\end{aligned}
\end{equation}
for all test functions $\chi$. Applying integration by parts to the derivative term gives
\begin{equation}\label{eq56}
    \int_{R_{\mathrm{in}}}^{R_{\mathrm{out}}} \overline{\chi}' v' r \, {\rm d}r
    = \overline{\chi(R_{\mathrm{out}})} R_{\mathrm{out}} v'(R_{\mathrm{out}}) - \overline{\chi(R_{\mathrm{in}})} R_{\mathrm{in}} v'(R_{\mathrm{in}})
    - \int_{R_{\mathrm{in}}}^{R_{\mathrm{out}}} \overline{\chi} (r v')' \, {\rm d}r.
\end{equation}
Substituting \eqref{eq56} into \eqref{eq55} we obtain
\begin{align}\label{eq57}
    &\operatorname{Re} \Biggl( \overline{\chi(R_{\mathrm{in}})} \bigl( \sqrt{R_{\mathrm{in}}^2 + q^2} \, v(R_{\mathrm{in}}) + \Lambda_q R_{\mathrm{in}} v'(R_{\mathrm{in}}) \bigr) \notag \\
    &\qquad + \overline{\chi(R_{\mathrm{out}})} \bigl( \sqrt{R_{\mathrm{out}}^2 + q^2} \, v(R_{\mathrm{out}}) - \Lambda_q R_{\mathrm{out}} v'(R_{\mathrm{out}}) \bigr) \notag \\
    &\qquad + \Lambda_q \int_{R_{\mathrm{in}}}^{R_{\mathrm{out}}} \overline{\chi} \Bigl( (r v')' - \Bigl(1 + \frac{q^2}{r^2}\Bigr) v r \Bigr) \, {\rm d}r \Biggr) = 0.
\end{align}

Because the test function $\chi$ is arbitrary, we may restrict it strictly to the interior so that $\chi(R_{\mathrm{in}}) = \chi(R_{\mathrm{out}}) = 0$. Under this restriction, \eqref{eq57} reduces to
\begin{equation}\label{eq61}
    \operatorname{Re} \Biggl( \Lambda_q \int_{R_{\mathrm{in}}}^{R_{\mathrm{out}}} \overline{\chi} \Bigl( (r v')' - \Bigl( 1 + \frac{q^2}{r^2} \Bigr) v r \Bigr) \, {\rm d}r \Biggr) = 0.
\end{equation}
Because $\chi$ is complex-valued, substituting ${\rm i}\chi$ for $\chi$ reveals that the imaginary part of this integral must also vanish. We thus conclude that
\begin{equation*}
    \int_{R_{\mathrm{in}}}^{R_{\mathrm{out}}} \overline{\chi} \Bigl( (r v')' - \Bigl( 1 + \frac{q^2}{r^2} \Bigr) v r \Bigr) \, {\rm d}r = 0
\end{equation*}
for all $\chi$. Hence, $v$ must satisfy the  Euler--Lagrange equation
\begin{equation}\label{eq59}
    (r v')' - \Bigl( 1 + \frac{q^2}{r^2} \Bigr) v r = 0 \quad \text{in } I_0,
\end{equation}
i.e.,
\begin{equation*}
    r^2 v'' + r v' - (r^2 + q^2) v = 0.
\end{equation*}
This is the modified Bessel equation, whose general solution is
\begin{equation}\label{eq58}
    v(r) = A_T I_{\alpha}(r) + B_T K_{\alpha}(r), \qquad \alpha \coloneqq |q|,
\end{equation}
where $A_T$ and $B_T$ are complex constants, and $I_{\alpha}$ and $K_{\alpha}$ are the modified Bessel functions of the first and second kind, respectively.

Returning to \eqref{eq57}, the validity of the Euler--Lagrange equation \eqref{eq59} implies that the interior integral vanishes identically. We are thus left solely with the boundary evaluations:
\begin{align*}
    \operatorname{Re} \Biggl( &\overline{\chi(R_{\mathrm{in}})} \bigl( \sqrt{R_{\mathrm{in}}^2 + q^2} \, v(R_{\mathrm{in}}) + \Lambda_q R_{\mathrm{in}} v'(R_{\mathrm{in}}) \bigr) \\
    &+ \overline{\chi(R_{\mathrm{out}})} \bigl( \sqrt{R_{\mathrm{out}}^2 + q^2} \, v(R_{\mathrm{out}}) - \Lambda_q R_{\mathrm{out}} v'(R_{\mathrm{out}}) \bigr) \Biggr) = 0.
\end{align*}
Because this identity holds for all complex-valued test functions $\chi$, we may prescribe the endpoint values $\chi(R_{\mathrm{in}})$ and $\chi(R_{\mathrm{out}})$ independently. By successively choosing the boundary pairs $(\chi(R_{\mathrm{in}}), \chi(R_{\mathrm{out}})) \in \{(1,0), ({\rm i},0), (0,1), (0,{\rm i})\}$, we deduce the following boundary conditions:
\begin{equation}\label{eq60}
\begin{cases}
    \sqrt{R_{\mathrm{in}}^2 + q^2} \, v(R_{\mathrm{in}}) + \Lambda_q R_{\mathrm{in}} v'(R_{\mathrm{in}}) = 0, \\[4pt]
    \sqrt{R_{\mathrm{out}}^2 + q^2} \, v(R_{\mathrm{out}}) - \Lambda_q R_{\mathrm{out}} v'(R_{\mathrm{out}}) = 0.
\end{cases}
\end{equation}

Substituting \eqref{eq58} into \eqref{eq60} yields a linear system for  $A_T$ and $B_T$:
\begin{equation*}
\begin{cases}
    A_T \bigl( \sqrt{R_{\mathrm{in}}^2 + q^2} \, I_{\alpha}(R_{\mathrm{in}}) + \Lambda_q R_{\mathrm{in}} I_{\alpha}'(R_{\mathrm{in}}) \bigr)
    + B_T \bigl( \sqrt{R_{\mathrm{in}}^2 + q^2} \, K_{\alpha}(R_{\mathrm{in}}) + \Lambda_q R_{\mathrm{in}} K_{\alpha}'(R_{\mathrm{in}}) \bigr) = 0, \\[6pt]
    A_T \bigl( \sqrt{R_{\mathrm{out}}^2 + q^2} \, I_{\alpha}(R_{\mathrm{out}}) - \Lambda_q R_{\mathrm{out}} I_{\alpha}'(R_{\mathrm{out}}) \bigr)
    + B_T \bigl( \sqrt{R_{\mathrm{out}}^2 + q^2} \, K_{\alpha}(R_{\mathrm{out}}) - \Lambda_q R_{\mathrm{out}} K_{\alpha}'(R_{\mathrm{out}}) \bigr) = 0.
\end{cases}
\end{equation*}
A nontrivial solution for the coefficients $A_T$ and $B_T$ exists if and only if the determinant of the associated coefficient matrix vanishes. This requires that
\begin{equation}\label{eq62}
\begin{vmatrix}
    \sqrt{R_{\mathrm{in}}^2 + q^2} \, I_{\alpha}(R_{\mathrm{in}}) + \Lambda_q R_{\mathrm{in}} I_{\alpha}'(R_{\mathrm{in}}) &
    \sqrt{R_{\mathrm{in}}^2 + q^2} \, K_{\alpha}(R_{\mathrm{in}}) + \Lambda_q R_{\mathrm{in}} K_{\alpha}'(R_{\mathrm{in}}) \\[8pt]
    \sqrt{R_{\mathrm{out}}^2 + q^2} \, I_{\alpha}(R_{\mathrm{out}}) - \Lambda_q R_{\mathrm{out}} I_{\alpha}'(R_{\mathrm{out}}) &
    \sqrt{R_{\mathrm{out}}^2 + q^2} \, K_{\alpha}(R_{\mathrm{out}}) - \Lambda_q R_{\mathrm{out}} K_{\alpha}'(R_{\mathrm{out}})
\end{vmatrix} = 0.
\end{equation}
This equation determines $\Lambda_q$.

Conversely, if $\widetilde{\Lambda}_q$ is a root of \eqref{eq62}, then there exists a nontrivial function $v \neq 0$ satisfying the boundary value problem
\begin{equation}\label{eq63}
\begin{cases}
    (r v')' - \Bigl( 1 + \frac{q^2}{r^2} \Bigr) v r = 0 & \text{in } I_0, \\[6pt]
    \sqrt{R_{\mathrm{in}}^2 + q^2} \, v(R_{\mathrm{in}}) + \widetilde{\Lambda}_q R_{\mathrm{in}} v'(R_{\mathrm{in}}) = 0 & \text{at } r = R_{\mathrm{in}}, \\[4pt]
    \sqrt{R_{\mathrm{out}}^2 + q^2} \, v(R_{\mathrm{out}}) - \widetilde{\Lambda}_q R_{\mathrm{out}} v'(R_{\mathrm{out}}) = 0 & \text{at } r = R_{\mathrm{out}}.
\end{cases}
\end{equation}
Multiplying the differential equation in \eqref{eq63} by $\overline{v}$ and integrating over the interval $I_0$ yields
\begin{equation*}
    \int_{R_{\mathrm{in}}}^{R_{\mathrm{out}}} (r v')' \overline{v} \, {\rm d}r - \int_{R_{\mathrm{in}}}^{R_{\mathrm{out}}} \Bigl( 1 + \frac{q^2}{r^2} \Bigr) |v|^2 r \, {\rm d}r = 0.
\end{equation*}
Applying integration by parts to the first term, we obtain
\begin{equation}\label{eq64}
    (r v' \overline{v}) \big|_{R_{\mathrm{in}}}^{R_{\mathrm{out}}} 
    = \int_{R_{\mathrm{in}}}^{R_{\mathrm{out}}} \Bigl( |v'|^2 + \frac{q^2}{r^2} |v|^2 + |v|^2 \Bigr) r \, {\rm d}r 
    = \frac{1}{2\pi} \mathcal{E}_q[v].
\end{equation}
Simultaneously, rearranging the boundary conditions in \eqref{eq63} yields
\begin{equation*}
\begin{aligned}
    \widetilde{\Lambda}_q R_{\mathrm{in}} v'(R_{\mathrm{in}}) &= - \sqrt{R_{\mathrm{in}}^2 + q^2} \, v(R_{\mathrm{in}}), \\
    \widetilde{\Lambda}_q R_{\mathrm{out}} v'(R_{\mathrm{out}}) &= \sqrt{R_{\mathrm{out}}^2 + q^2} \, v(R_{\mathrm{out}}).
\end{aligned}
\end{equation*}
Multiplying these identities by $\overline{v(R_{\mathrm{in}})}$ and $\overline{v(R_{\mathrm{out}})}$, respectively, and subtracting the former from the latter, we obtain
\begin{equation}\label{eq65}
    \widetilde{\Lambda}_q (r v' \overline{v}) \big|_{R_{\mathrm{in}}}^{R_{\mathrm{out}}} = \frac{1}{2\pi} \mathcal{B}_q[v].
\end{equation}
Combining \eqref{eq64} and \eqref{eq65} reveals that
\begin{equation*}
    \widetilde{\Lambda}_q \mathcal{E}_q[v] = \mathcal{B}_q[v].
\end{equation*}

Because $v \neq 0$ implies $\mathcal{E}_q[v] > 0$, we conclude that
\begin{equation*}
    \widetilde{\Lambda}_q = \frac{\mathcal{B}_q[v]}{\mathcal{E}_q[v]} \le \sup_{w \neq 0} \frac{\mathcal{B}_q[w]}{\mathcal{E}_q[w]} = \Lambda_q.
\end{equation*}
Therefore, $\Lambda_q(R_{\mathrm{in}}, R_{\mathrm{out}})$ is precisely the largest root of \eqref{eq62}. The optimal constant corresponding to the $q$-th Fourier mode is thus given by
\begin{equation*}
    \widetilde{C}_{T,q}^2 = \Lambda_q(R_{\mathrm{in}}, R_{\mathrm{out}}).
\end{equation*}
Consequently, the global optimal constant for the equivalent trace norm is
\begin{equation*}
    \widetilde{C_T}^2 = \sup_{q \in \mathbb{Z}} \widetilde{C}_{T,q}^2 = \sup_{q \in \mathbb{Z}} \Lambda_q(R_{\mathrm{in}}, R_{\mathrm{out}}).
\end{equation*}

By analogous reasoning, these results extend to dimensions $d \ge 3$, where the quantity $\Lambda_q$ also depends on the dimension.

Therefore, for a shell domain $D_k$ with inner radius $R_{\mathrm{in}}^{(k)}$ and outer radius $R_{\mathrm{out}}^{(k)}$, the modified optimal trace constant $\widetilde{C_T}^{(k)}$ is given by
\begin{align}
\label{eq:trace_constant}
\left(\widetilde{C_T}^{(k)}\right)^2 = 
\begin{cases} 
\sup\limits_{q \in \mathbb{Z}} \Lambda_q^{(k)}(R_{\mathrm{in}}^{(k)}, R_{\mathrm{out}}^{(k)}), & d = 2, \\
\sup\limits_{q \in \mathbb{N}_0} \Lambda_q^{(k)}(R_{\mathrm{in}}^{(k)}, R_{\mathrm{out}}^{(k)}, d), & d \ge 3,
\end{cases}
\end{align}
where, for $d \ge 3$, the index $q$ denotes the degree of the corresponding spherical harmonics.

Furthermore, there exist positive constants $m_1$ and $m_2$, depending only on the radii $R_{\mathrm{in}}^{(k)}, R_{\mathrm{out}}^{(k)}$ and the dimension $d$, such that
\begin{equation*}
m_1 C_T^{(k)} \le \widetilde{C_T}^{(k)} \le m_2 C_T^{(k)}.
\end{equation*}
This double inequality implies that $C_T^{(k)}$ and $\widetilde{C_T}^{(k)}$ are comparable. Consequently, any estimate derived for $\widetilde{C_T}^{(k)}$ applies to the original optimal trace constant $C_T^{(k)}$ up to a constant factor.


\end{document}